\documentclass[a4paper]{article}
\usepackage{amssymb}
\usepackage{amsmath}
\usepackage{amsthm}
\usepackage[USenglish]{babel}
\usepackage[latin1]{inputenc}
\usepackage[colorlinks=true]{hyperref}
\usepackage{authblk}
\newtheorem{teo}{Theorem}[section]
\newtheorem{lemma}[teo]{Lemma}
\newtheorem{prop}[teo]{Proposition}
\newtheorem{cor}[teo]{Corollary}
\newtheorem{oss}[teo]{Remark}
\newcommand{\dd}{\displaystyle}
\newcommand{\de}{\partial}
\newcommand{\mb}{\mathbb}
\newcommand{\epsi}{\varepsilon}
\newcommand{\pphi}{\varphi}
\newcommand{\ttilde}{\widetilde}
\newcommand{\convdeb}{\rightharpoonup}
\DeclareMathOperator{\esssup}{esssup}
\DeclareMathOperator{\spt}{spt}
\DeclareMathOperator{\dive}{div}
\DeclareMathOperator{\loc}{loc}

\begin{document}

\title{Remarks on the Moser-Trudinger inequality}
\author{Luca Battaglia\thanks{S.I.S.S.A/I.S.A.S, Via Bonomea 265, 34136 Trieste (Italy) - lbatta@sissa.it}\and Gabriele Mancini\thanks{S.I.S.S.A/I.S.A.S, Via Bonomea 265, 34136 Trieste (Italy) - gmancini@sissa.it\\The authors are supported by the FIRB project {\em Analysis and Beyond}, by the PRIN {\em Variational Methods and Nonlinear PDE's} and by the Mathematics Department at the University of Warwick.}}
\date{}

\maketitle

\begin{abstract}
\noindent We extend the Moser-Trudinger inequality
$$\sup_{u\in W^{1,N}_0(\Omega),\int_\Omega|\nabla u(x)|^Ndx\le1}\int_\Omega\left(e^{\alpha_N|u(x)|^\frac{N}{N-1}}-\sum_{j=0}^{N-2}\frac{\alpha_N^j|u(x)|^\frac{jN}{N-1}}{j!}\right)dx<+\infty$$
to any Euclidean domain satisfying Poincaré's inequality
$$\lambda_1(\Omega):=\inf_{0\not\equiv u\in W^{1,N}_0(\Omega)}\frac{\int_\Omega|\nabla u(x)|^Ndx}{\int_\Omega|u(x)|^Ndx}>0.$$
We find out that the same equivalence does not hold in general for conformal metrics on the unit ball, showing counterexamples.\\
We also study the existence of extremals for the Moser-Trudinger inequalities for unbounded domains, proving it for the infinite planar strips $\dd{\Omega:=\mb R\times(-1,1)}$.
\end{abstract}

\section*{Introduction}
This paper is concerned with extensions of the Moser-Trudinger inequality.\\
Trudinger \cite{trudi} discovered that for any bounded open domain $\dd{\Omega\subset\mb R^N}$ one has
\begin{equation}
\label{eq:mosertru}
\sup_{u\in W^{1,N}_0(\Omega),\int_\Omega|\nabla u(x)|^Ndx\le1}\int_\Omega e^{\alpha|u(x)|^\frac{N}{N-1}}dx<+\infty
\end{equation}
for some exponent $\dd{\alpha>0}$; Moser \cite{moser} then showed that $\dd{\eqref{eq:mosertru}}$ is true if and only if
$$\alpha\le\alpha_N:=N\omega_{N-1}^\frac{1}{N-1}$$
where
$$\omega_{N-1}=\frac{N\pi^\frac{N}2}{\Gamma\left(\frac{N}2+1\right)}$$
is the $\dd{N-1}$-dimensional measure of the unit sphere $\dd{\mb S^{N-1}}$.\\
This work is focused on extensions of the inequality to infinite-measure domains. Clearly, in this case the integral in $\dd{\eqref{eq:mosertru}}$ is infinite, since the integrand is greater than $\dd{1}$; moreover one should remove the first terms of the power series expansion of $\dd{e^{\alpha_N|u|^\frac{N}{N-1}}}$, that in general are not summable in $\dd{W^{1,N}_0(\Omega)}$. Precisely, we will consider the integral of
\begin{equation}
\label{eq:phi}
\Phi(u):=e^{\alpha_N|u|^\frac{N}{N-1}}-\sum_{j=0}^{N-2}\frac{\alpha_N^j|u|^\frac{jN}{N-1}}{j!}.
\end{equation}
We will say that the Moser-Trudinger inequality holds for an unbounded domain $\dd{\Omega\subset\mb R^N}$ if
\begin{equation}
\label{eq:mosertr2}
\sup_{u\in W^{1,N}_0(\Omega),\int_\Omega|\nabla u(x)|^Ndx\le1}\int_\Omega\Phi(u(x))dx<+\infty.
\end{equation}
More generally, considering an $\dd{N}$-dimensional Riemannian manifold $\dd{(M,g)}$, we will say that the Moser-Trudinger inequality holds if
$$\sup_{u\in W^{1,N}_0(M,g),\int_M|\nabla_gu(x)|^NdV_g(x)\le1}\int_M\Phi(u(x))dV_g(x)<+\infty.$$
Since the first power in the series expansion of $\dd{\Phi(u)}$ is $\dd{\frac{\alpha_N^{N-1}}{(N-1)!}|u|^N}$, the Moser-Trudinger inequality on $\dd{\Omega}$ implies Poincaré's inequality
\begin{equation}
\label{eq:poincare}
\lambda_1(\Omega):=\inf_{0\not\equiv u\in W^{1,N}_0(\Omega)}\frac{\int_\Omega|\nabla u(x)|^Ndx}{\int_\Omega|u(x)|^Ndx}>0.
\end{equation}
Gi. Mancini and Sandeep \cite{mansand} studied the problem on the unit disc $\dd{M=B_1(0)\subset\mb R^2}$ endowed with a conformal metric $\dd{g_\rho=\rho g_e}$, and they found that the Moser-Trudinger inequality holds for $\dd{(M,g_\rho)}$ if and only if $\dd{g_\rho}$ is bounded by the hyperbolic metric, namely
$$\rho(x)\le\frac{C}{\left(1-|x|^2\right)^2}\quad\forall\;x\in M\quad\text{for some }C>0.$$
A particular case is given, through the Riemann map, by Euclidean simply connected domains: in this case, they showed that the Moser-Trudinger inequality is actually equivalent to Poincaré's inequality $\dd{\eqref{eq:poincare}}$.\\
In this paper we prove that the same equivalence holds for any Euclidean domain $\dd{\Omega\subset\mb R^N}$, even for $\dd{N\ge3}$ (Theorem $\dd{\ref{teo:gabri}}$).\\
We will also investigate whether the equivalence between the Moser-Trudinger and the Poincaré inequality even holds for conformal metrics on the unit ball $\dd{B_1(0)\subset\mb R^N}$, and we discover that this does not occur.\\
We build counterexamples of conformal metrics such that Poincaré's inequality holds but Moser-Trudinger's does not. In these examples we show that the highest exponent for exponential integrability on non-compact manifolds may be any number smaller than $\dd{\alpha_N}$; moreover, exponential integrability might not occur at all, and even summability of higher powers fails for suitable metrics (Theorems $\dd{\ref{teo:funodue}}$ and $\dd{\ref{teo:esisteg}}$).\\
In the last part of this paper, we study the existence of extremal functions for the Moser-Trudinger inequality: Carleson-Chang \cite{carchang}, Flucher \cite{flucher} and K.-C. Lin \cite{kclin} solved this problem for any Euclidean bounded domain, but nothing had been done, up to our knowledge, for unbounded domains yet.\\
Here we give the first existence result for unbounded domains, precisely for the strip $\dd{\Omega=\mb R\times(-1,1)\subset\mb R^2}$ (Theorem $\dd{\ref{teo:striscia}}$).\\
When passing from bounded to unbounded domains, the main difficulty is that P.-L. Lions' concentration-compactness principle \cite{lions} is no longer true in its original form: non-compact sequences may ``vanish'' at infinity, besides ``concentrating''; however, the symmetry of $\dd{\Omega}$ with the respect to both axes allows to exclude both vanishing and concentration whereas the Riemann map between $\dd{\Omega}$ and the unit disc allows to exclude concentration, as for bounded domains.\\
Section one is about the Moser-Trudinger inequality on unbounded euclidean domains, section two concerns the Moser-Trudinger inequality on the unit ball endowed with a conformal metric, and finally in section three we study the existence of extremals on the strip.

\section{The Moser-Trudinger inequality on unbounded Euclidean domains}
In this section, we investigate the Euclidean domains where the Moser-Trudinger inequality holds; as mentioned before, the first result concerning unbounded domains was given by Gi. Mancini and Sandeep \cite{mansand}:
\begin{teo}
\label{teo:mansand}$\dd{}$\\
Let $\dd{\Omega\subset\mb R^2}$ be a simply connected open domain.\\
Then, the following conditions are equivalent:
\begin{enumerate}
\item The Moser-Trudinger inequality holds for $\dd{\Omega}$, that is
$$\sup_{u\in H^1_0(\Omega),\int_\Omega|\nabla u(x)|^2dx\le1}\int_\Omega\left(e^{4\pi u(x)^2}-1\right)dx<+\infty.$$
\item Poincaré's inequality holds for $\dd{\Omega}$, that is 
$$\lambda_1(\Omega)=\inf_{0\not\equiv u\in H^1_0(\Omega)}\frac{\int_\Omega|\nabla u(x)|^2dx}{\int_\Omega u(x)^2dx}>0.$$
\item The inradius of $\dd{\Omega}$
$$r(\Omega):=\sup\{R>0\;:\;\exists\;x\in\Omega\text{ such that }B_R(x)\subset\Omega\}$$
is finite.
\end{enumerate}
\end{teo}
We stress that removing the hypothesis of simple connectedness, the finiteness of inradius is weaker than the other two statements; a counterexemple satisfying $\dd{r(\Omega)<+\infty}$ and $\dd{\lambda_1(\Omega)=0}$ is shown indeed in \cite{mansand}.\\
The equivalence between Moser-Trudinger and Poincaré's inequalities, instead, will be now extended to any Euclidean domain.
\begin{teo}
\label{teo:gabri}$\dd{}$\\
Let $\dd{\Omega\subset\mb R^N}$ be an open domain and let $\dd{\Phi}$ be as in $\dd{\eqref{eq:phi}}$.\\
Then, the Moser-Trudinger inequality holds for $\dd{\Omega}$ if and only if Poincaré's inequality $\dd{\eqref{eq:poincare}}$ does.
\end{teo}
An essential tool in the proof of Theorem $\dd{\ref{teo:gabri}}$ is the Schwarz symmetrization.\\
The Schwarz symmetrization of a non-negative $\dd{u\in W^{1,N}_0(\Omega)}$ is a non-negative radially nonincreasing $\dd{u^*\in W^{1,N}_0\left(\mb R^N\right)}$ that is equidistributed with $\dd{u}$, namely
$$|\{u>t\}|=\left|\left\{u^*>t\right\}\right|\quad\text{for any }t\ge0,$$
hence (see for instance \cite{kesavan}) it holds
$$\int_{\mb R^N}f\left(u^*(x)\right)dx=\int_\Omega f(u(x))dx\quad\text{for any Borel }f:\mb R\to\mb R$$
and in particular for $\dd{f(u)=|u|^p}$ and $\dd{f=\Phi}$, so for any function $\dd{u\in W^{1,N}_0(\Omega)}$, the Moser-Trudinger functional has the same value in $\dd{u}$ and $\dd{|u|^*}$.
Moreover, due to P\'olya-Szeg\H o inequality
$$\int_{\mb R^N}|\nabla u^*(x)|^Ndx\le\int_\Omega|\nabla u(x)|^Ndx$$
the condition on the Dirichlet integral in $\dd{\eqref{eq:mosertr2}}$ also holds for $\dd{|u|^*}$ if it does for $\dd{u}$; so, the Moser-Trudinger inequality can be proved without loss of generality just on Schwarz-symmetrized functions.\\
For the definition of the Schwarz symmetrization and a detailed list of its properties see for instance \cite{baern}.
\begin{oss}$\dd{}$\\
Any admissible $\dd{u}$ for the supremum in Theorem $\dd{\ref{teo:gabri}}$ can be seen as a function in $\dd{W^{1,N}_0\left(\mb R^N\right)}$ satisfying
\begin{equation}
\label{eq:tanaka}
\int_{\mb R^N}|\nabla u(x)|^Ndx\le 1\quad\text{and}\quad\int_{\mb R^N}|u(x)|^Ndx\le C
\end{equation}
with
$$C=\frac{1}{\lambda_1(\Omega)}$$
however, if one extends the set of admissible functions to any $\dd{u\in W^{1,N}_0\left(\mb R^N\right)}$ satisfying $\dd{\eqref{eq:tanaka}}$, the supremum is infinite for any $\dd{C>0}$, even if one considers just radially nonincreasing functions: taking a sequence of Moser-like functions
$$u_k(x)=\frac{1}{\omega_{N-1}^\frac{1}N}\left\{\begin{array}{ll}k^{N-1}&\text{if }|x|\le ke^{-k^N}\\\frac{\log\frac{k}{|x|}}k&\text{if }ke^{-k^N}<|x|\le k\\0&\text{if }|x|>k\end{array}\right.$$
one has
$$\int_{\mb R^N}|\nabla u_k(x)|^Ndx=\omega_{N-1}\int_{ke^{-k^N}}^k\left(\frac{1}{\omega_{N-1}^\frac{1}N}\frac{1}{k\rho}\right)^N\rho^{N-1}d\rho=\frac{1}{k^N}\int_{ke^{-k^N}}^k\frac{d\rho}\rho=1$$
whereas
$$\int_{\mb R^N}|u_k(x)|^Ndx\le\int_{B_k(0)}\left(\frac{1}{\omega_{N-1}^\frac{1}N}\frac{\log\frac{k}{|x|}}k\right)^Ndx=\frac{1}{k^N}\int_0^k\left(\log\frac{k}\rho\right)^N\rho^{N-1}d\rho=$$
$$=\int_0^{+\infty}t^Ne^{-Nt}dt=\frac{N!}{N^{N+1}}$$
and
$$\int_{\mb R^N}\Phi(u_k(x))dx\ge\int_{B_{ke^{-k^N}}(0)}\left(e^{Nk^N}-\sum_{j=0}^{N-2}\frac{N^jk^{jN}}{j!}\right)dx=$$
$$=k^Ne^{-Nk^N}\left(e^{Nk^N}-\sum_{j=0}^{N-2}\frac{N^jk^{jN}}{j!}\right)dx\underset{k\to+\infty}\to+\infty.$$
Thus, setting $\dd{v_k(x)=u_k\left(\left(\frac{N!}{N^{N+1}C}\right)^\frac{1}Nx\right)}$, the new functions verify
$$\int_{\mb R^N}|\nabla v_k(x)|^Ndx=\int_{\mb R^N}|\nabla u_k(y)|^Ndy=1,$$
$$\int_{\mb R^N}|v_k(x)|^Ndx=\frac{N^{N+1}C}{N!}\int_{\mb R^N}|u_k(y)|^Ndy\le C$$
and
$$\int_{\mb R^N}\Phi(v_k(x))dx=\frac{N^{N+1}C}{N!}\int_{\mb R^N}\Phi(u_k(y))dy\underset{k\to+\infty}\to+\infty.$$
On the other hand, replacing $\dd{\alpha_N}$ with any smaller exponent gives uniform boundedness of the Moser-Trudinger functional even in this case, as proved by Adachi and Tanaka \cite{adatan}.
\end{oss}
The proof of Theorem $\dd{\ref{teo:gabri}}$ will need some lemmas, the first of which extends the ``radial lemma'' from Berestycki and P.-L. Lions \cite{berlions}:
\begin{lemma}
\label{lemma:radial}$\dd{}$\\
If $\dd{0\le u\in W^{1,N}_0\left(\mb R^N\right)}$ is a radially nonincreasing function, then
$$u(x)\le\left(\frac{N}{\omega_{N-1}}\right)^\frac{1}N\frac{\|u\|_{L^N\left(\mb R^N\right)}}{|x|}.$$
\end{lemma}
\begin{proof}$\dd{}$\\
If $\dd{u(x)=U(|x|)}$, then
$$r^NU(r)=NU(r)\int_0^r\rho^{N-1}d\rho\le N\int_0^rU(\rho)\rho^{N-1}d\rho\le$$
$$\le N\left(\int_0^rU(\rho)^N\rho^{N-1}d\rho\right)^\frac{1}N\left(\int_0^r\rho^{N-1}d\rho\right)^\frac{N-1}N\le$$
$$\le N\left(\int_0^{+\infty}U(\rho)^N\rho^{N-1}d\rho\right)^\frac{1}N\frac{r^{N-1}}{N^\frac{N-1}N}=\left(\frac{N}{\omega_{N-1}}\right)^\frac{1}Nr^{N-1}\|u\|_{L^N\left(\mb R^N\right)},$$
hence the claim.
\end{proof}
\begin{lemma}
\label{lemma:polya}$\dd{}$\\
If $\dd{0\le u\in W^{1,N}_0(\Omega)}$ and $\dd{\ttilde x\in\mb R^N}$ is such that $\dd{u^*(\ttilde x)=t}$, then
$$\int_{B_{|\ttilde x|}(0)}|\nabla u^*(x)|^Ndx\le\int_{\{u>t\}}|\nabla u(x)|^Ndx.$$
\end{lemma}
\begin{proof}$\dd{}$\\
If $\dd{t=u^*(\ttilde x)=0}$, then it is the standard P\'olya-Szeg\H o inequality, whereas if $\dd{t=\esssup_\Omega u}$, then the inequality is trivial since both sides are zero. Now let $\dd{t\in(0,\esssup_\Omega u)}$; then $\dd{\ttilde R:=\inf\{|x|:u^*(x)=t\}}$ verifies
$$\int_{B_{|\ttilde x|}(0)}|\nabla u^*(x)|^Ndx=\int_{B_{\ttilde R}(0)}|\nabla u^*(x)|^Ndx$$
and $\dd{B_{\ttilde R}(0)=\{u^*>t\}=\{u>t\}^*}$, where the last asterisk denotes the Schwarz symmetrization of a set.\\
The function
$$v(x)=\left\{\begin{array}{ll}u(x)-t&\text{if }u(x)>t\\0&\text{otherwise}\end{array}\right.$$
verifies
$$v^*(x)=\left\{\begin{array}{ll}u^*(x)-t&\text{in }B_{\ttilde R}(0)\\0&\text{otherwise}\end{array}\right.,$$
hence one can conclude, through P\'olya-Szeg\H o inequality:
$$\int_{B_{|\ttilde x|}(0)}|\nabla u^*(x)|^Ndx=\int_{B_{\ttilde R}(0)}|\nabla u^*(x)|^Ndx=\int_{\mb R^N}|\nabla v^*(x)|^Ndx\le$$
$$\le\int_\Omega|\nabla v(x)|^Ndx=\int_{\{u>t\}}|\nabla u(x)|^Ndx.$$
\end{proof}
\begin{lemma}
\label{lemma:gabri}$\dd{}$\\
If $\dd{\lambda_1(\Omega)>0}$, $\dd{0\le u\in W^{1,N}_0(\Omega)}$ and $\dd{\int_\Omega|\nabla u(x)|^Ndx=1}$, then
$$u^*(\ttilde x)^N\le\frac{N}{\omega_{N-1}|\ttilde x|^N\lambda_1(\Omega)}\left(1-\int_{B_{|\ttilde x|}(0)}|\nabla u^*(x)|^Ndx\right).$$
\end{lemma}
\begin{proof}$\dd{}$\\
If $\dd{u^*(\ttilde x)=0}$ the statement is trivial, otherwise setting $\dd{t:=u^*(\ttilde x)>0}$, $\dd{\overline{R}:=}$\\$\dd{=\sup\{|x|:u^*(x)=t\}}$ and
$$w(x)=\left\{\begin{array}{ll}t&\text{if }u(x)>t\\u(x)&\text{otherwise}\end{array}\right.$$
one gets
$$w^*(x)=\left\{\begin{array}{ll}t&\text{in }B_{\overline R}(0)\\u^*(x)&\text{otherwise}\end{array}\right..$$
Thus, applying Lemma $\dd{\ref{lemma:radial}}$,
$$u^*(\ttilde x)^N=t^N=w^*(\ttilde x)^N\le\frac{N}{\omega_{N-1}|\ttilde x|^N}\int_{\mb R^N}{w^*(x)}^Ndx=$$  
$$=\frac{N}{\omega_{N-1}|\ttilde x|^N}\int_\Omega w(x)^Ndx\le\frac{N}{\omega_{N-1}|\ttilde x|^N\lambda_1(\Omega)}\int_\Omega|\nabla w(x)|^Ndx\le$$
$$\le\frac{N}{\omega_{N-1}|\ttilde x|^N\lambda_1(\Omega)}\int_{\{u\le t\}}|\nabla u(x)|^Ndx\le$$
$$\le\frac{N}{\omega_{N-1}|\ttilde x|^N\lambda_1(\Omega)}\left(1-\int_{\{u>t\}}|\nabla u(x)|^Ndx\right)\le$$
$$\le\frac{N}{\omega_{N-1}|\ttilde x|^N\lambda_1(\Omega)}\left(1-\int_{B_{|\ttilde x|}(0)}|\nabla u^*(x)|^Ndx\right),$$
where the last inequality follows from Lemma $\dd{\ref{lemma:polya}}$.
\end{proof}
Now we have all the tools necessary to prove the main result of this section.
\begin{proof}[Proof of Theorem $\dd{\ref{teo:gabri}}$]$\dd{}$\\
The Moser-Trudinger inequality trivially implies Poincaré's inequality.\\
For the other implication, symmetrization allows to consider for the supremum only symmetrized functions $\dd{u\in W^{1,N}_0\left(\mb R^N\right)}$ satisfying $$\int_{\mb R^N}|\nabla u(x)|^Ndx\le1.$$
If $\dd{\int_{B_1(0)}|\nabla u(x)|^Ndx=1}$, then $\dd{u\in W^{1,N}_0(B_1(0))}$ hence the classical Moser--Trudinger inequality on $\dd{B_1(0)}$ gives
$$\int_{\mb R^N}\Phi(u(x))dx=\int_{B_1(0)}\Phi(u(x))dx\le C(N,\Omega).$$
Supposing instead $\dd{\int_{B_1(0)}|\nabla u(x)|^Ndx<1}$, one can estimate the Moser-Trudinger functional on $\dd{B_1(0)^c}$ through Lemma $\dd{\ref{lemma:radial}}$.\\
In fact one has
$$u(x)\le\left(\frac{N}{\omega_{N-1}}\right)^\frac{1}{N}\frac{\|u\|_{L^N(\Omega)}}{|x|}\le\left(\frac{N}{\omega_{N-1}\lambda_1(\Omega)}\right)^\frac{1}N,$$
hence, using the estimate
\begin{equation}
\label{eq:stimaesp}
e^x-\sum_{j=0}^{N-2}\frac{x^j}{j!}=\sum_{j=N-1}^{+\infty}\frac{x^j}{j!}\le\sum_{j=N-1}^{+\infty}\frac{x^j}{(j-N+1)!(N-1)!}=\frac{x^{N-1}}{(N-1)!}e^x
\end{equation}
for any $\dd{x\ge0}$ , $\dd{2\le N\in\mb N}$ one gets
$$\int_{B_1(0)^c}\Phi(u(x))dx\le\int_{B_1(0)^c}\frac{\alpha_N^{N-1}}{(N-1)!}|u(x)|^Ne^{\alpha_N|u(x)|^\frac{N}{N-1}}dx\le$$
$$\le\frac{\alpha_N^{N-1}}{(N-1)!}e^{\alpha_N\left(\frac{N}{\omega_{N-1}\lambda_1(\Omega)}\right)^\frac{1}{N-1}}\int_{B_1(0)^c}|u(x)|^Ndx\le$$
\begin{equation}
\label{eq:fuori}
\le\frac{\alpha_N^{N-1}}{(N-1)!}\frac{e^{\alpha_N\left(\frac{N}{\omega_{N-1}\lambda_1(\Omega)}\right)^\frac{1}{N-1}}}{\lambda_1(\Omega)}=:C_1(N,\Omega).
\end{equation}
To estimate the integral over $\dd{B_1(0)}$, it suffices to consider the function 
$$v(x)=\left\{\begin{array}{cc}u(x)-u(1)&\text{in }B_1(0)\\0&\text{otherwise}\end{array}\right.$$
where $\dd{u(1)}$ indicates, with a little abuse of notation, the trace of $\dd{u}$ on the boundary of $\dd{B_1(0)}$; the elementary inequality
$$(A+B)^p\le A^p+pA^{p-1}B+|B|^p$$
for $\dd{A\ge0}$, $\dd{B\ge-A}$ and $\dd{p\in[0,2]}$ and a weighted Young inequality give, for any $\dd{\epsi>0}$,
$$u(x)^\frac{N}{N-1}\le v(x)^\frac{N}{N-1}(1+\epsi)+{u(1)}^\frac{N}{N-1}\left(1+\frac{1}{\epsi^\frac{1}{N-1}(N-1)^\frac{1}{N-1}}\right)=$$
$$=w(x)^\frac{N}{N-1}+{u(1)}^\frac{N}{N-1}\left(1+\frac{1}{((N-1)\epsi)^\frac{1}{N-1}}\right),$$
where $\dd{w(x):=v(x)(1+\epsi)^\frac{N-1}N}$.\\
Choosing
$$\epsi=\frac{1-\left(\int_{B_1(0)}|\nabla u(x)|^Ndx\right)^\frac{1}{N-1}}{\left(\int_{B_1(0)}|\nabla u(x)|^Ndx\right)^\frac{1}{N-1}}$$
one gets
$$\int_{B_1(0)}|\nabla w(x)|^Ndx=(1+\epsi)^{N-1}\int_{B_1(0)}|\nabla v(x)|^Ndx=$$
$$=\left(1+\frac{1-\left(\int_{B_1(0)}|\nabla u(x)|^Ndx\right)^\frac{1}{N-1}}{\left(\int_{B_1(0)}|\nabla u(x)|^Ndx\right)^\frac{1}{N-1}}\right)^{N-1}\int_{B_1(0)}|\nabla u(x)|^Ndx=1,$$
hence we can apply the Moser-Trudinger inequality on $\dd{B_1(0)}$ to obtain
$$\int_{B_1(0)}e^{\alpha_Nw(x)^\frac{N}{N-1}}dx\le C(N,\Omega).$$
Moreover, from Lemma $\dd{\ref{lemma:gabri}}$,
$$\frac{u(1)^\frac{N}{N-1}}{\epsi^\frac{1}{N-1}}=u(1)^\frac{N}{N-1}\frac{\left(\int_{B_1(0)}|\nabla u(x)|^Ndx\right)^\frac{1}{N-1}}{\left(1-\left(\int_{B_1(0)}|\nabla u(x)|^Ndx\right)^\frac{1}{N-1}\right)^\frac{1}{N-1}}\le$$
$$\le\frac{u(1)^\frac{N}{N-1}}{\left(1-\left(\int_{B_1(0)}|\nabla u(x)|^Ndx\right)^\frac{1}{N-1}\right)^\frac{1}{N-1}}\le$$
$$\le\frac{N}{\omega_{N-1}\lambda_1(\Omega)}\frac{\left(1-\int_{B_1(0)}|\nabla u(x)|^Ndx\right)^\frac{1}{N-1}}{\left(1-\left(\int_{B_1(0)}|\nabla u(x)|^Ndx\right)^\frac{1}{N-1}\right)^\frac{1}{N-1}}\le C_2(N,\Omega),$$
where last passage follows from boundedness of $\dd{t\to\frac{1-t}{1-t^\frac{1}{N-1}}}$ on $\dd{[0,1)}$.\\
Hence
$$\int_{B_1(0)}\Phi(u(x))dx\le\int_{B_1(0)}e^{\alpha_N{|u(x)|}^\frac{N}{N-1}}dx\le$$
$$\le e^{\alpha_N{u(1)}^\frac{N}{N-1}\left(1+\frac{1}{((N-1)\epsi)^\frac{1}{N-1}}\right)}\int_{B_1(0)}e^{\alpha_Nw(x)^\frac{N}{N-1}}dx$$
\begin{equation}
\label{eq:dentro}
\le e^{\left(\frac{N}{\omega_{N-1}\lambda_1(\Omega)}\right)^\frac{1}{N-1}+\frac{C_2(N,\Omega)}{(N-1)^\frac{1}{N-1}}}C(N,\Omega)=:C_3(N,\Omega),
\end{equation}
and the conclusion follows from $\dd{\eqref{eq:fuori}}$ and $\dd{\eqref{eq:dentro}}$.
\begin{oss}$\dd{}$\\
With few modifications, the proof of Theorem $\dd{\ref{teo:gabri}}$ can be extended to a Moser-like functional with a different number of term removed, that is
$$\Phi_k(u):=e^{\alpha_N|u|^\frac{N}{N-1}}-\sum_{j=0}^{k-1}\frac{\alpha_N^j|u|^\frac{jN}{N-1}}{j!}.$$
Even in this case, if the first power of the functional is controlled by the $\dd{L^N}$ norm of the gradient, then all the functional is, precisely
$$\lambda_{1,k}(\Omega):=\inf_{0\not\equiv u\in W^{1,N}_0(\Omega)}\frac{\int_\Omega|\nabla u(x)|^Ndx}{\int_\Omega|u(x)|^\frac{kN}{N-1}dx}>0$$
$$\Updownarrow$$
$$\sup_{u\in W^{1,N}_0(\Omega),\int_\Omega|\nabla u(x)|^Ndx\le1}\int_\Omega\Phi_k(u(x))dx<+\infty.$$
\end{oss}
\end{proof}

\section{The Moser-Trudinger inequality for conformal metrics}
A natural question to ask is whether Theorems $\dd{\ref{teo:mansand}}$ and $\dd{\ref{teo:gabri}}$ can be extended to other mainfolds besides Euclidean domains, such as the unit ball in $\dd{\mb R^N}$ endowed with a conformal metric.\\
A characterization of the metrics where the Moser-Trudinger inequality holds was given by Gi. Mancini and Sandeep in the $\dd{2}$-dimensional case \cite{mansand} and it was later extended by themselves and Tintarev \cite{mansanti} to any dimension.
\begin{teo}
\label{teo:conformi}$\dd{}$\\
Let $\dd{g_\rho=\rho(x)g_e}$ be a conformal metric on $\dd{B=B_1(0)\subset\mb R^N}$.\\
Then the following conditions are equivalent:
\begin{enumerate}
\item $\dd{g_\rho}$ is bounded by the hyperbolic metric, that is
$$\exists\;C>0\quad\text{such that }\rho(x)\le\frac{C}{\left(1-|x|^2\right)^2}\quad\forall\;x\in B.$$
\item The Moser-Trudinger inequality holds for the metric $\dd{g}$, that is
$$\sup_{u\in W^{1,N}_0(B,g_\rho),\int_B|\nabla_{g_\rho}u(x)|^NdV_{g_\rho}(x)\le1}\int_B\Phi(u(x))dV_{g_\rho}(x)<+\infty.$$
\end{enumerate}
\end{teo}
In view of Theorem $\dd{\ref{teo:conformi}}$, the question can be rewritten as: are the conformal metrics on $\dd{B}$ such that Poincaré's inequality holds all and only the ones which are bounded by the hyperbolic metric $\dd{g_h}$?\\
As usual, Poincaré's inequality is implied by the Moser-Trudinger inequality, and therefore by the boundedness of the conformal factor with respect to the hyperbolic metric. Moreover, a partial converse can be shown easily:
\begin{prop}
\label{prop:rho}$\dd{}$\\
Let $\dd{g_\rho=\rho(x)g_e}$ be a conformal metric defined on the unit ball $\dd{B=B_1(0)\subset\mb R^N}$ such that
\begin{equation}
\label{eq:rhobordo}
\lim_{x\to\ttilde x}\zeta(x)=+\infty\quad\text{for some }\ttilde x\in\de B
\end{equation}
where $\dd{\zeta(x)=\frac{\left(1-|x|^2\right)^2}4\rho(x)}$ is the conformal factor with respect to the hyperbolic metric.\\
Then, $\dd{\lambda_1(B,g_\rho)=0}$.
\end{prop}
\begin{oss}$\dd{}$\\
Condition $\dd{\eqref{eq:rhobordo}}$ is actually stronger than the unboundedness of $\dd{\zeta}$, that can be expressed as
$$\limsup_{x\to\ttilde x}\zeta(x)=+\infty\quad\text{for some }\ttilde x\in\de B.$$
\end{oss}
To prove Proposition $\dd{\ref{prop:rho}}$, and also later on, one requires some conformal diffeomorphisms, called M\"obius maps, that extend in higher dimension the well-known biholomorphisms of the complex unit disc $\dd{\pphi_a(z)=\frac{z+a}{1+\overline az}}$.\\
The following lemma, whose proof is a simple calculation, lists the main properties of the M\"obius maps (see \cite{rat} for more details):
\begin{lemma}
\label{lemma:phi}$\dd{}$\\
For any $\dd{a\in B}$ and $\dd{x\in\mb R^N\backslash\left\{-\frac{a}{|a|^2}\right\}}$, the map
$$\pphi_a(x)=\frac{\left(1-|a|^2\right)x+\left(|x|^2+2\langle a,x\rangle+1\right)a}{|a|^2|x|^2+2\langle a,x\rangle+1}$$
has the following properties:
\begin{enumerate}
\item $\dd{\pphi_a(0)=a}$.
\item $\dd{\pphi_a}$ is a diffeomorphism between $\dd{\mb R^N\backslash\left\{-\frac{a}{|a|^2}\right\}}$ and $\dd{\mb R^N\backslash\left\{\frac{a}{|a|^2}\right\}}$.
\item $\dd{\pphi_a(B)=B}$.
\item $\dd{\pphi_a}$ is conformal.
\item $\dd{\pphi_a}$ is a hyperbolic isometry.
\item $\dd{\pphi_a(B_R(0))=B_{\frac{R\left(1-|a|^2\right)}{1-R^2|a|^2}}\left(\frac{\left(1-R^2\right)}{1-R^2|a|^2}a\right)}$ $\dd{\forall\;R\in(0,1)}$.
\end{enumerate}
\end{lemma}
\begin{proof}[Proof of Proposition $\dd{\ref{prop:rho}}$]$\dd{}$\\
A sequence $\dd{u_k\in W^{1,N}_0(B,g_\rho)}$ such that
$$\int_B|\nabla_{g_\rho}u_k(x)|^NdV_{g_\rho}(x)\le1\quad\text{and}\quad\int_B|u_k(x)|^NdV_{g_\rho}(x)\underset{k\to+\infty}\to+\infty$$
can be built in this way: given a function $\dd{0\not\equiv u\in W^{1,N}_0(B,g_{\rho})}$ such that
$$\spt u\subset B_{\frac{1}2}(0)\quad\text{and}\quad\int_B|\nabla_{g_\rho}u(y)|^NdV_{g_\rho}(y)\le1$$
and a sequence $\dd{x_k\underset{k\to+\infty}\to\ttilde x}$, it suffices to put $\dd{u_k=u\circ\pphi_{x_k}^{-1}}$, where $\dd{\pphi_{x_k}}$ are as in Lemma $\dd{\ref{lemma:phi}}$.\\
Conformal invariance of the Dirichlet integral yields
$$\int_B|\nabla_{g_\rho}u_k(x)|^NdV_{g_\rho}(x)=\int_B|\nabla_{g_\rho}u(y)|^NdV_{g_\rho}(y)\le1$$
whereas, to work with the $\dd{L^N}$ norm, one observes that for for every $\dd{M>0}$ the set $\dd{\{\zeta\ge M\}}$ is a neighborhood of $\dd{\ttilde x}$, hence, due to the properties of $\dd{\pphi_{x_k}}$, it contains $\dd{\pphi_{x_k}\left(B_\frac{1}2(0)\right)}$ for $\dd{k}$ sufficiently large.\\
Thus, being the $\dd{\pphi_{x_k}}$'s hyperbolic isometries,
$$\int_B|u_k(x)|^NdV_{g_\rho}(x)=\int_{\pphi_{x_k}\left(B_\frac{1}2(0)\right)}|u_k(x)|^N\zeta(x)^\frac{N}2dV_{g_h}(x)\ge$$
$$\ge M^\frac{N}2\int_{\pphi_{x_k}\left(B_\frac{1}2(0)\right)}|u_k(x)|^NdV_{g_h}(x)=M^\frac{N}2\int_{B_\frac{1}2(0)}|u(y)|^NdV_{g_h}(y)$$
so, being $\dd{M}$ arbitrary, one has
$$\int_B|u_k(x)|^NdV_{g_\rho}(x)\underset{k\to+\infty}\to+\infty$$
and so $\dd{\lambda_1(B,g_\rho)=0}$.
\end{proof}
However, despite this result, in general Poincaré's inequality does not imply exponential integrability up to the critical exponent, that is Theorems $\dd{\ref{teo:mansand}}$ and $\dd{\ref{teo:gabri}}$ cannot be extended in this context.\\
Moreover, not only exponential but also polynomial integrability is not ensured; actually, for any couple of nonlinearities satisfying some basic properties one can build a metric, which attains higher and higher values on small balls accumulating on $\dd{\de B}$, such that the former nonlinearity is uniformly summable and the latter is not.\\
Precisely, the result we obtained is the following:
\begin{teo}
\label{teo:funodue}$\dd{}$\\
Let $\dd{f_1,f_2\in C(\mb R,\mb R)}$ be even, positive, nondecreasing on $\dd{[0,+\infty)}$ and satisfying
\begin{enumerate}
\item $\dd{f_1(u)\le C\Phi^\alpha(u)\text{ for some }C>0,\alpha<\alpha_N}$
where
$$\Phi^\alpha(u):=e^{\alpha|u|^\frac{N}{N-1}}-\sum_{j=0}^{N-2}\frac{\alpha^j|u|^\frac{jN}{N-1}}{j!}=\Phi\left(\left(\frac{\alpha}{\alpha_N}\right)^\frac{N-1}Nu\right).$$
\item We have
$$\frac{f_1(u)}{f_2(u)}\underset{u\to+\infty}\to0.$$
\item We have $$\liminf_{R\to0}\frac{R^N\ttilde{f_1}(R)}{\int_0^R\ttilde{f_1}(R)\rho^{N-1}d\rho}>0\quad\text{where}\quad\ttilde f(R):=f\left(\frac{\left(\log\frac{1}R\right)^\frac{N-1}N}{\omega_{N-1}^\frac{1}N}\right).$$
\end{enumerate}
Then, there exists a conformal metric $\dd{g_\rho=\rho(x)g_e}$ such that
$$\sup_{u\in W^{1,N}_0(B,g_\rho),\int_B|\nabla_{g_\rho}u(x)|^NdV_{g_\rho}(x)\le1}\int_Bf_1(u(x))dV_{g_\rho}(x)<+\infty$$
and
$$\sup_{u\in W^{1,N}_0(B,g_\rho),\int_B|\nabla_{g_\rho}u(x)|^NdV_{g_\rho}(x)\le1}\int_Bf_2(u(x))dV_{g_\rho}(x)=+\infty.$$
\end{teo}
\begin{oss}$\dd{}$\\
Notice that if we choose the functions $\dd{f_i}$'s such that $\dd{f_1(u)\ge C|u|^N}$ and $\dd{f_2(u)\le C\Phi(u)}$ we get that Poincaré's inequality holds whereas the Moser-Trudinger inequality does not, thus showing that Theorems $\dd{\ref{teo:mansand}}$ and $\dd{\ref{teo:gabri}}$ are no longer true in this setting.
\end{oss}
To prove Theorem $\dd{\ref{teo:funodue}}$, we will use a radial estimate which plays a similar role of Lemma $\dd{\ref{lemma:radial}}$; the following Lemma is by Tintarev \cite{tintarev}, who used it to give another proof of the hyperbolic Moser-Trudinger inequality.
\begin{lemma}
\label{lemma:radialip}$\dd{}$\\
Let $\dd{u\in W^{1,N}_0(B,g_\rho)}$ be a radially nonincreasing function. Then
$$u(x)\le\frac{\|\nabla_{g_\rho}u\|_{L^N(B,g_\rho)}}{\omega_{N-1}^\frac{1}N}\left(\log\frac{1}{|x|}\right)^\frac{N-1}N.$$
\end{lemma}
\begin{proof}$\dd{}$\\
If $\dd{u(x)=U(|x|)}$, then
$$U(r)=\left|-\int_r^1U'(\rho)d\rho\right|\le\left(\int_r^1|U'(\rho)|^N\rho^{N-1}d\rho\right)^\frac{1}N\left(\int_r^1\frac{d\rho}\rho\right)^\frac{N-1}N\le$$
$$\le\left(\int_0^1|U'(\rho)|^N\rho^{N-1}d\rho\right)^\frac{1}N\left(\int_r^1\frac{d\rho}\rho\right)^\frac{N-1}N=\frac{\|\nabla_{g_\rho}u\|_{L^N(B,g_\rho)}}{\omega_{N-1}^\frac{1}N}\left(\log\frac{1}r\right)^\frac{N-1}N.$$
\end{proof}
\begin{proof}[Proof of Theorem $\dd{\ref{teo:funodue}}$]$\dd{}$\\
The metric will be built as $$g_\rho=\rho(x)g_e=\zeta(x)g_h=\left(1+\sum_{k=1}^{+\infty}\eta_k(x)\right)^\frac{2}Ng_h,$$
with
\begin{enumerate}
\item $\dd{\eta_k\in C^\infty_0(B_k)\text{ for }B_k:=\pphi_{x_k}(B_{R_k}(0))}$
\item $\dd{0\le\eta_k\le a_k}$
\item $\dd{\eta_k\ge\frac{a_k}2\text{ on }A_k:=\pphi_{x_k}\left(B_{\frac{R_k}2}(0)\right)}$
\end{enumerate}
where the $\dd{\pphi_{x_k}}$'s are as in Lemma $\dd{\ref{lemma:phi}}$, the $\dd{x_k}$ are taken such that $\dd{\spt(\eta_k)}$ are pairwise disjoint, for instance
$$|x_k|\ge\frac{R_k+R_{k-1}+|x_{k-1}|+R_kR_{k-1}|x_{k-1}|}{1+R_kR_{k-1}+R_k|x_{k-1}|+R_{k-1}|x_{k-1}|}$$
and $\dd{a_k\underset{k\to+\infty}\nearrow+\infty}$, $\dd{R_k\underset{k\to+\infty}\searrow0}$ as $\dd{k\to+\infty}$ will be chosen later in dependence of $\dd{f_1,f_2}$.\\
For any function $\dd{u\in W^{1,N}_0(B,g_\rho)}$ satisfying $\dd{\int_B|\nabla_{g_\rho}u(x)|^NdV_{g_\rho}(x)\le1}$ we set $\dd{v_k:=u\circ\pphi_{x_k}}$ and denote as $\dd{v_k^*}$ its Schwarz symmetrization. Moser-Trudinger inequality on the hyperbolic disc, Lemma $\dd{\ref{lemma:radialip}}$ and the properties of $\dd{f_1}$ yield
$$\int_Bf_1(u(x))dV_{g_\rho}(x)=\int_Bf_1(u(x))dV_{g_h}(x)+\sum_{k=1}^{+\infty}\int_{B_k}f_1(u(x))\eta_k(x)dV_{g_h}(x)\le$$
$$\le C+\sum_{k=1}^{+\infty}\int_{B_{R_k}(0)}f_1(v_k(y))\eta_k(\pphi_{x_k}(y))dV_{g_h}(y)\le$$
$$\le C+\sum_{k=1}^{+\infty}a_k\left(\frac{2}{1-R_k^2}\right)^N\int_{B_{R_k}(0)}f_1(v_k(y))dV_{g_e}(y)\le$$
$$\le C+\left(\frac{2}{1-R_1^2}\right)^N\sum_{k=1}^{+\infty}a_k\int_{B_{R_k}(0)}f_1\left(\left|v_k\right|^*(y)\right)dV_{g_e}(y)\le$$
$$\le C+C_1\sum_{k=1}^{+\infty}a_k\int_{B_{R_k}(0)}\ttilde{f_1}(|y|)dV_{g_e}(y)=$$
\begin{equation}
\label{eq:funo}
=C+C_1\omega_{N-1}\sum_{k=1}^{+\infty}a_k\int_0^{R_k}\ttilde{f_1}(\rho)\rho^{N-1}d\rho\le C_2+C_3\sum_{k=1}^{+\infty}a_kR_k^N\ttilde{f_1}(R_k).
\end{equation}
On the other hand, for $\dd{f_2}$, one takes a sequence of Moser functions
$$w_k(x)=\frac{1}{\omega_{N-1}^\frac{1}N}\left\{\begin{array}{ll}\left(\log\frac{1}{R_k}\right)^\frac{N-1}N&\text{if }|x|<\frac{R_k}2\\\frac{\log\frac{1}{2|x|}}{\left(\log\frac{1}{R_k}\right)^\frac{1}N}&\text{if }\frac{R_k}2<|x|\le\frac{1}2\\0&\text{if }|x|>\frac{1}2\end{array}\right.$$
and evaluates $\dd{f_2}$'s integral on $\dd{u_k:=w_k\circ\pphi_{x_k}^{-1}}$:
conformal invariance gives
$$\int_B|\nabla_{g_\rho}u_k(x)|^NdV_{g_\rho}(x)=\int_B|\nabla_{g_\rho}w_k(y)|^NdV_{g_\rho}(y)=1$$
whereas the fact that $\dd{\pphi_{x_k}}$ preserves hyperbolic measure and the construction of $\dd{\eta_k}$ allow to write
$$\int_Bf_2(u_k(x))dV_{g_\rho}(x)=\int_Bf_2(w_k(y))\zeta(\pphi_{x_k}(y))^\frac{N}2dV_{g_h}(y)\ge$$
$$\ge\int_{B_\frac{R_k}2(0)}f_2(w_k(y))\left(\zeta(\pphi_{x_k}(y))^\frac{N}2-1\right)dV_{g_h}(y)\ge$$
\begin{equation}
\label{eq:fdue}
\ge2^N\frac{a_k}2\int_{B_\frac{R_k}2(0)}f_2(w_k(y))dV_{g_e}(y)=\frac{\omega_{N-1}}{2N}a_kR_k^N\ttilde{f_2}(R_k).
\end{equation}
Taking
$$a_k=\frac{1}{R_k^N\sqrt{\ttilde{f_1}(R_k)\ttilde{f_2}(R_k)}}$$
the quantity $\dd{\eqref{eq:fdue}}$ becomes
$$\frac{\omega_{N-1}}{2N}\sqrt{\frac{\ttilde{f_2}(R_k)}{\ttilde{f_1}(R_k)}}\underset{k\to+\infty}\to+\infty$$
whereas the addend in the sum $\dd{\eqref{eq:funo}}$ becomes
$$\sqrt{\frac{\ttilde{f_1}(R_k)}{\ttilde{f_2}(R_k)}}\underset{k\to+\infty}\to0$$
that is summable if $\dd{R_k}$ is chosen properly.
\end{proof}
The last condition in the statement of Theorem $\dd{\ref{teo:funodue}}$ seems tricky but it is actually satisfied by most elementary functions which satisfy the limitation on the growth.\\
The following result is a simple calculus exercise:
\begin{prop}
\label{prop:funodue}$\dd{}$\\
For any $\dd{p\ge N}$, $\dd{\alpha\in(0,\alpha_N)}$, $\dd{\beta>0}$, $\dd{q\in\left(0,\frac{N}{N-1}\right)}$, the following functions satisfy the third condition of Theorem $\dd{\ref{teo:funodue}}$:
$$f_1(u)=\max\left\{|u|^N,|u|^p\right\},\quad f_1(u)=\max\left\{|u|^N,\frac{|u|^p}{1+\log(1+|u|)}\right\},$$
$$f_1(u)=e^{\beta|u|^q}-\sum_{j=0}^{\left[\frac{N}q\right]}\frac{\beta^j|u|^{jq}}{j!},\quad f_1(u)=\Phi^\alpha(u),\quad f_1(u)=\frac{\Phi^\alpha(u)}{1+|u|}.$$
\end{prop}
\begin{proof}$\dd{}$\\
Notice that the third hypothesis of Theorem $\dd{\ref{teo:funodue}}$ only depends on the behavior of $\dd{f_1}$ around infinity.\\
One can easily see that both $\dd{R^N\ttilde{f_1}(R)}$ and $\dd{\int_0^R\ttilde{f_1}(R)\rho^{N-1}d\rho}$ tend to $\dd{0}$ as $\dd{R\to0}$, hence the $\dd{\liminf}$ can be computed as a limit using l'H\^opital's rule and a change of variable:
$$\lim_{R\to0}\frac{R^N\ttilde f_1(R)}{\int_0^R\ttilde f_1(\rho)\rho^{N-1}d\rho}=\lim_{R\to0}\frac{NR^{N-1}\ttilde f_1(R)+R^N{\ttilde f_1}'(R)}{R^{N-1}\ttilde f_1(R)}=N+\lim_{R\to0}\frac{R{\ttilde f_1}'(R)}{\ttilde f_1(R)}=$$
$$=N+\lim_{R\to0}\frac{Rf_1'\left(\frac{\left(\log\frac{1}R\right)^\frac{N-1}N}{\omega_{N-1}^\frac{1}N}\right)}{f_1\left(\frac{\left(\log\frac{1}R\right)^\frac{N-1}N}{\omega_{N-1}^\frac{1}N}\right)}\left(-\frac{N-1}N\frac{1}R\left(\frac{\left(\log\frac{1}R\right)^\frac{-1}N}{\omega_{N-1}^\frac{1}N}\right)\right)=$$
$$=N-\frac{N-1}{\alpha_N}\lim_{u\to+\infty}\frac{{f_1}'(u)}{u^\frac{1}{N-1}f_1(u)}.$$
The existence and positiveness of the last limit is a simple calculation.
\end{proof}
\begin{teo}
\label{teo:esisteg}$\dd{}$\\
For any $\dd{p>N,\alpha\in(0,\alpha_N)}$ there exists a metric $\dd{g_\rho=\rho(x)g_e}$ that satisfies one of the following conditions:
\begin{enumerate}
\item For all $\dd{q>N}$,
$$\sup_{u\in W^{1,N}_0(B,g_\rho),\int_B|\nabla_{g_\rho}u(x)|^NdV_{g_\rho}(x)\le1}\int_B|u(x)|^qdV_{g_\rho}(x)=+\infty.$$
\item We have
$$\sup_{u\in W^{1,N}_0(B,g_\rho),\int_B|\nabla_{g_\rho}u(x)|^NdV_{g_\rho}(x)\le1}\int_B|u(x)|^rdV_{g_\rho}(x)<+\infty$$
if and only if $\dd{r\in[N,p)}$.
\item We have
$$\sup_{u\in W^{1,N}_0(B,g_\rho),\int_B|\nabla_{g_\rho}u(x)|^NdV_{g_\rho}(x)\le1}\int_B|u(x)|^rdV_{g_\rho}(x)<+\infty$$
if and only if $\dd{r\in[N,p]}$.
\item For all $\dd{r\ge N}$,
$$\sup_{u\in W^{1,N}_0(B,g_\rho),\int_B|\nabla_{g_\rho}u(x)|^NdV_{g_\rho}(x)\le1}\int_B|u(x)|^rdV_{g_\rho}(x)<+\infty\;\forall\;r\ge$$,
and for all $\dd{\beta>0}$,
$$\sup_{u\in W^{1,N}_0(B,g_\rho),\int_B|\nabla_{g_\rho}u(x)|^NdV_{g_\rho}(x)\le1}\int_B\Phi^\beta(u(x))dV_{g_\rho}(x)=+\infty$$
\item We have
$$\sup_{u\in W^{1,N}_0(B,g_\rho),\int_B|\nabla_{g_\rho}u(x)|^NdV_{g_\rho}(x)\le1}\int_B\Phi^\gamma(u(x))dV_{g_\rho}(x)<+\infty$$
if and only if $\dd{\gamma\in(0,\alpha)}$.
\item We have
$$\sup_{u\in W^{1,N}_0(B,g_\rho),\int_B|\nabla_{g_\rho}u(x)|^NdV_{g_\rho}(x)\le1}\int_B\Phi^\gamma(u(x))dV_{g_\rho}(x)<+\infty$$
if and only if $\dd{\gamma\in(0,\alpha]}$.
\item We have
$$\sup_{u\in W^{1,N}_0(B,g_\rho),\int_B|\nabla_{g_\rho}u(x)|^NdV_{g_\rho}(x)\le1}\int_B\Phi^\gamma(u(x))dV_{g_\rho}(x)<+\infty$$
if and only if $\dd{\gamma\in(0,\alpha_N)}$.
\end{enumerate}
\end{teo}
\begin{oss}$\dd{}$\\
Notice that for $\dd{\alpha\ge\alpha_N}$ Theorem $\dd{\ref{teo:conformi}}$ implies that, if $\dd{g_\rho}$ is as in the proof of Theorem $\dd{\ref{teo:funodue}}$,
$$\sup_{u\in W^{1,N}_0(B,g_\rho),\int_B|\nabla_{g_\rho}u(x)|^NdV_{g_\rho}(x)\le1}\int_B\Phi^\alpha(u(x))dV_{g_\rho}(x)=+\infty,$$
moreover, one cannot take $\dd{f_1(u)=|u|^p}$ for $\dd{p<N}$ because one does not have $\dd{|u|^p\le C\Phi^\alpha(u)}$ when $\dd{|u|}$ is small; actually, it holds
$$\sup_{u\in W^{1,N}_0(B,g_\rho),\int_B|\nabla_{g_h}u(x)|^NdV_{g_h}(x)\le1}\int_B|u(x)|^pdV_{g_h}(x)=+\infty$$
hence these powers are not even summable on the disc endowed with the metric $\dd{g_\rho}$.
\end{oss}
\begin{proof}$\dd{}$\\
All but the last point follow from applying Theorem $\dd{\ref{teo:funodue}}$ with suitable $\dd{f_1,f_2}$, that can be chosen as in Proposition $\dd{\ref{prop:funodue}}$:
\begin{enumerate}
\item It suffices to take
$$f_1(u)=|u|^N\quad\quad\text{and}\quad\quad f_2(u)=\max\left\{e^{-\frac{1}u},|u|^N\log|u|\right\}$$
since $\dd{|u|^q\ge C(q)f_2(u)}$ for any $\dd{q>N}$.
\item It suffices to take
$$f_1(u)=\max\left\{|u|^N,\frac{|u|^p}{1+\log(1+|u|)}\right\}\quad\text{and}\quad f_2(u)=\max\left\{e^{-\frac{1}u},|u|^p\right\}$$
since $\dd{|u|^r\le C(r)f_1(u)}$ for any $\dd{r\in[N,p)}$ and $\dd{|u|^q\ge C(q)f_2(u)}$ for any $\dd{q\ge p}$.
\item It suffices to take
$$f_1(u)=\max\left\{|u|^N,|u|^p\right\}\quad\quad\text{and}\quad\quad f_2(u)=\max\left\{e^{-\frac{1}u},|u|^p\log|u|\right\}$$
since $\dd{|u|^r\le C(r)f_1(u)}$ for any $\dd{r\in[N,p]}$ and $\dd{|u|^q\ge C(q)f_2(u)}$ for any $\dd{q>p}$.
\item It suffices to take
$$f_1(u)=e^{|u|}-\sum_{j=0}^{N-1}\frac{|u|^j}{j!}\quad\quad\text{and}\quad\quad f_2(u)=e^{2|u|}-\sum_{j=0}^{N-1}\frac{2^j|u|^j}{j!}$$
since $\dd{|u|^r\le C(r)f_1(u)}$ for any $\dd{r\ge N}$ and $\dd{\Phi^\beta(u)\ge C(\beta)f_2(u)}$ for any $\dd{\beta>0}$.
\item It suffices to take
$$f_1(u)=\frac{\Phi^\alpha(u)}{1+|u|}\quad\quad\text{and}\quad\quad f_2(u)=\Phi^\alpha(u)$$
since $\dd{\Phi^\gamma(u)\le C(\gamma)f_1(u)}$ for any $\dd{\gamma\in(0,\alpha)}$.
\item It suffices to take
$$f_1(u)=\Phi^\alpha(u)\quad\text{and}\quad f_2(u)=(1+|u|)\Phi^\alpha(u)$$
since $\dd{\Phi^\beta(u)\ge C(\beta)f_2(u)}$ for any $\dd{\beta>\alpha}$.
\item For the last point, it suffices to show a metric $\dd{g_\rho}$ such that
$$f_1(u)=\frac{\Phi(u)}{1+|u|^\frac{2N}{N-1}}$$
is uniformly integrable on $\dd{W^{1,N}_0(B,g_\rho)}$; however, $\dd{f_1}$ does not satisfy all the hypotheses of Theorem $\dd{\ref{teo:funodue}}$, since it has the same asymptotic behavior of
$$f_0(u)=\frac{e^{\alpha_N|u|^\frac{N}{N-1}}}{|u|^\frac{2N}{N-1}}$$
and, being $$\ttilde{f_0}(R)=\frac{1}{\omega_{N-1}^\frac{2}{N-1}R^N\left(\log\frac{1}R\right)^2},$$
one has
$$\lim_{R\to0}\frac{R^N\ttilde{f_0}(R)}{\int_0^R\ttilde{f_0}(\rho)\rho^{N-1}d\rho}=\lim_{R\to0}\frac{\frac{1}{\omega_{N-1}^\frac{2}{N-1}\left(\log\frac{1}R\right)^2}}{\int_0^R\frac{d\rho}{\omega_{N-1}^\frac{2}{N-1}\rho\left(\log\frac{1}\rho\right)^2}d\rho}=\lim_{R\to0}\frac{1}{\log\frac{1}R}=0.$$
However, one can argue as in the proof of Theorem $\dd{\ref{teo:funodue}}$, building the metric in the same way and choosing later $\dd{a_k}$ and $\dd{R_k}$:
$$\sup_{u\in W^{1,N}_0(B,g_\rho),\int_B|\nabla_{g_\rho}u(x)|^NdV_{g_\rho}(x)\le1}\int_Bf(u(x))dV_{g_\rho}(x)\le$$
$$\le C+C_2\sum_{k=1}^{+\infty}a_k\int_0^{R_k}\ttilde f(\rho)\rho^{N-1}d\rho\le$$
$$\le C+C_2\sum_{k=1}^{+\infty}a_k\int_0^{R_k}\ttilde{f_0}(\rho)\rho^{N-1}d\rho=C+C_2\sum_{k=1}^{+\infty}\frac{a_k}{\log\frac{1}{R_k}}$$
that converges choosing for instance $\dd{a_k=k}$ and $\dd{R_k=e^{-k^3}}$.
\end{enumerate}
\end{proof}

\section{Extremals for the Moser-Trudinger inequality on strips}
The last section is devoted to the problem of extremal functions for the Moser-Trudinger inequality.\\
As mentioned before, the existence of extremals has been already proved for any bounded	set $\dd{\Omega\subset\mb R^N}$:	
\begin{teo}
\label{teo:limitati}$\dd{}$\\
Let $\dd{\Omega\subset\mb R^N}$ be a bounded open domain.\\
Then, there exists a function $\dd{\ttilde u\in W^{1,N}_0(\Omega)}$ with $\dd{\int_\Omega|\nabla\ttilde u(x)|^Ndx\le1}$ and
$$\int_\Omega e^{\alpha_N|\ttilde u(x)|^\frac{N}{N-1}}dx=\sup_{u\in W^{1,N}_0(\Omega),\int_\Omega|\nabla u(x)|^Ndx\le1}\int_\Omega e^{\alpha_N|u(x)|^\frac{N}{N-1}}dx.$$
\end{teo}
This has been first proved by Carleson and Chang \cite{carchang}, when $\dd{\Omega=B_R(\ttilde x)}$ is a ball, then by Flucher \cite{flucher} for any planar domain, and finally by K.-C. Lin \cite{kclin} in the general case. Here, a first existence result is given for an unbounded domain, precisely $\dd{\Omega=\mb R\times(-1,1)}$:
\begin{teo}
\label{teo:striscia}$\dd{}$\\
If $\dd{\Omega=\mb R\times(-1,1)\subset\mb R^2}$, then there exists a function $\dd{\ttilde u\in H^1_0(\Omega)}$ such that\\
$\dd{\int_\Omega|\nabla u(x)|^2dx\le1}$ and
$$\int_\Omega\left(e^{4\pi{\ttilde u(x)}^2}-1\right)dx=\sup_{u\in H^1_0(\Omega),\int_\Omega|\nabla u(x)|^2dx\le1}\int_\Omega\left(e^{4\pi u(x)^2}-1\right)dx.$$
\end{teo}
\begin{oss}$\dd{}$\\
The Moser-Trudinger inequality for $\dd{\Omega}$ follows from Theorem $\dd{\ref{teo:mansand}}$, since $\dd{r(\Omega)=1}$.
\end{oss}
A basic tool for the proof of Theorem $\dd{\ref{teo:limitati}}$ is the concentration-compactness principle by P.-L. Lions \cite{lions}; it states that a bounded sequence in $\dd{W^{1,N}_0(\Omega)}$ can either be compact for the Moser-Trudinger functional or concentrate at some point:
\begin{teo}
\label{teo:conccomp}$\dd{}$\\
Let $\dd{\Omega\subset\mb R^N}$ be a bounded open domain and let $\dd{u_k\in W^{1,N}_0(\Omega)}$ be a sequence satisfying $\dd{\int_\Omega|\nabla u_k(x)|^Ndx\le1}$.\\
Then, up to a subsequence, one of the following alternatives holds true:
\begin{enumerate}
\item As $\dd{k\to+\infty}$, $\dd{u_k\underset{k\to+\infty}\convdeb0}$ and $\dd{|\nabla u_k|^Ndx\underset{k\to+\infty}\convdeb\delta_{\ttilde x}}$ for some $\dd{\ttilde x\in\overline\Omega}$, and
$$\lim_{k\to+\infty}\int_{\Omega\backslash B_\epsi(\ttilde x)}\left(e^{\alpha_N|u_k(x)|^\frac{N}{N-1}}-1\right)dx=0\quad\forall\;\epsi>0.$$
\item As $\dd{k\to+\infty}$, $\dd{u_k\underset{k\to+\infty}\convdeb u}$ and $\dd{e^{\alpha_N|u_k|^\frac{N}{N-1}}\underset{k\to+\infty}\to e^{\alpha_N|u|^\frac{N}{N-1}}}$ in $\dd{L^1(\Omega)}$.
\end{enumerate}
The sequence $\dd{\{u_k\}}$ will be called ``concentrating'' at the point $\dd{\ttilde x}$  in the first case, and ``compact'' in the latter.
\end{teo}
In view of this, the proof of Theorem $\dd{\ref{teo:limitati}}$ will follow by showing that maximizing sequences for the Moser-Trudinger inequality cannot concentrate, and this is what was actually done by \cite{carchang,flucher,kclin}.\\
When dealing with general domains, one can use a slight modification of the previous result:
\begin{teo}
\label{teo:concomp2}$\dd{}$\\
Let $\dd{\Omega\subset\mb R^N}$ be an open domain with the property $\dd{\lambda_1(\Omega)>0}$ and let $\dd{u_k\in W^{1,N}_0(\Omega)}$ be a sequence satisfying $\dd{\int_\Omega|\nabla u_k(x)|^Ndx\le1}$.\\
Then, up to a subsequence, one of the following alternatives holds true:
\begin{enumerate}
\item As $\dd{k\to+\infty}$, $\dd{u_k\underset{k\to+\infty}\convdeb0	}$ and $\dd{|\nabla u_k|^Ndx\underset{k\to+\infty}\convdeb\delta_{\ttilde x}}$ for some $\dd{\ttilde x\in\overline\Omega}$, and
$$\lim_{k\to+\infty}\int_{K\backslash B_\epsi(\ttilde x)}\left(e^{\alpha_N|u_k(x)|^\frac{N}{N-1}}-1\right)dx=0\quad\forall\;\epsi>0,\;\forall K\subset\Omega\text{ bounded}.$$
\item As $\dd{k\to+\infty}$, $\dd{u_k\underset{k\to+\infty}\convdeb u}$ and $\dd{e^{\alpha_N|u_k|^\frac{N}{N-1}}\underset{k\to+\infty}\to e^{\alpha_N|u|^\frac{N}{N-1}}}$ in $\dd{L^1_{\loc}(\Omega)}$.
\end{enumerate}
\end{teo}
\begin{proof}$\dd{}$\\
We first notice that, in view of Theorem $\dd{\ref{teo:gabri}}$, Moser-Trudinger inequality holds in $\dd{\Omega}$, hence the integral of $$e^{\alpha_N|u|^\frac{N}{N-1}}$$
is finite in every bounded open sets contained in $\dd{\Omega}$.\\
We consider an increasing sequence of bounded sets $\dd{\{K_j\}_{j=1}^\infty}$ with
$$\bigcup_{j=1}^\infty K_j=\Omega$$
and apply Theorem $\dd{\ref{teo:conccomp}}$ to each of these sets, starting from $\dd{K_1}$.\\
Suppose that, applying the Concentration-Compactness Theorem on $\dd{K_1}$, one finds that $\dd{u_k}$ concentrates in some $\dd{\ttilde x\in\overline\Omega}$; then, since $\dd{\mu\left(\overline\Omega\right)\le1}$, the same alternative will occur when the Concentration-Compactness principle will be applied on the other sets $\dd{K_j}$, hence one gets the first alternative in Theorem $\dd{\ref{teo:concomp2}}$.\\
On the other hand, if we get compactness on $\dd{K_1}$, we continue to apply $\dd{\ref{teo:conccomp}}$ on the sets $\dd{K_j}$: if we have concentration in one of these sets, we argue as before getting the first alternative; otherwise, we will find compactness of the functional on any bounded set, that is the second alternative in $\dd{\ref{teo:concomp2}}$.
\end{proof}
\begin{oss}$\dd{}$\\
Notice that in Theorem $\dd{\ref{teo:concomp2}}$ the convergence in $\dd{L^1_{\loc}(\Omega)}$ actually means convergence in $\dd{L^1(K)}$ for any bounded set $\dd{K\subset\Omega}$, even if $\dd{\overline K\not\subset\Omega}$.\\
This little abuse of notation will be also done in the rest of this paper; the same holds when convergence in $\dd{L^p_{\loc}(\Omega)}$ will be considered for some other $\dd{p>1}$.
\end{oss}
Intuitively, the main difference between the Theorems $\dd{\ref{teo:conccomp}}$ and $\dd{\ref{teo:concomp2}}$ is that the mass might `disappear' at the infinity; taking for instance, on the aforementioned strip $\dd{\Omega=\mb R\times(0,1)}$,
$$0\not\equiv u\in H^1_0(\Omega)\text{ with }\spt u\Subset\Omega\quad\text{and}\quad u_k(x)=u(x-(k,0)),$$
on every compact subset of $\dd{\Omega}$, $\dd{u_k\equiv0}$ definitively, so $$u_k\underset{k\to+\infty}\convdeb0\quad\text{and}\quad|\nabla u_k|^2\underset{k\to+\infty}\convdeb0$$
whereas
$$\int_\Omega\left(e^{4\pi u_k(x)^2}-1\right)dx=\int_\Omega\left(e^{4\pi u(x)^2}-1\right)dx\ne0\quad\forall\;k\in\mb N.$$
This phenomenon is known as ``vanishing''. Moreover there is a fourth scenario to be considered, the so-called dichotomy, that is when just a part of the mass is vanishing; as an example, one may take $\dd{v_k=(1-\theta)u_k+\theta w_k}$ with $\dd{u_k}$ as before, $\dd{w_k}$ compact for the Moser-Trudinger functional and $\dd{\theta\in(0,1)}$.\\
The strategy to exclude vanishing and dichotomy is trying to restrict the set of admissible functions for the supremum to those which have some symmetries, hence satisfy some uniform decay estimates similar to Lemmas $\dd{\ref{lemma:radial}}$ and $\dd{\ref{lemma:radialip}}$.\\
The symmetry of the strip $\dd{\Omega=\mb R\times(-1,1)\subset\mb R^2}$ with respect to both axes allows to perform a trick similar to Schwarz symmetrization: it is possible to apply twice a one-dimensional symmetrization, with respect to each axes, to any $\dd{u\in H^1_0(\Omega)}$, and the symmetrized function is still defined on $\dd{\Omega}$.\\
To be precise, for any fixed $\dd{x_2\in(-1,1)}$ one defines
$$\mu_{u,1}(t)=|\{x_1\in\mb R:u(x_1,x_2)>t\}|$$
and
$$u^{*,1}(x_1,x_2)=\inf\{t\in\mb R:\mu_{u,1}(t)\le2|x_1|\};$$
in the same way, one puts
$$\mu_{u,2}(t)=|\{x_2\in(-1,1):u(x_1,x_2)>t\}|$$
and
$$u^{*,1}(x_1,x_2)=\inf\{t\in\mb R:\mu_{u,1}(t)\le2|x_2|\},$$
and finally sets
$$u^{*,\Omega}=(u^{*,1})^{*,2};$$
since $\dd{u^{*,\Omega}}$ is obtained by applying twice a Steiner symmetrization, some good properties of Schwarz symmetrization still hold:
\begin{enumerate}
\item $\dd{u^{*,\Omega}(x_1,x_2)}$ is even in both $\dd{x_1}$ and $\dd{x_2}$, that is
$$u^{*,\Omega}(x_1,x_2)=u^{*,\Omega}(-x_1,x_2)=u^{*,\Omega}(x_1,-x_2).$$
\item $\dd{u^{*,\Omega}(x_1,x_2)}$ is nonincreasing in both variables for nonnegative $\dd{x_1}$ and $\dd{x_2}$, that is
$$|x_1|\le|X_1|,|x_2|\le|X_2|\Rightarrow u(x_1,x_2)\ge u(X_1,X_2).$$
\item For any Borel $\dd{f:\mb R\to\mb R}$, it holds
$$\int_\Omega f(u(x))dx=\int_\Omega f\left(u^{*,\Omega}(x)\right)dx.$$
\item If $\dd{0\le u\in H^1_0(\Omega)}$, then $\dd{0\le u^{*,\Omega}\in H^1_0(\Omega)}$ and
$$\int_\Omega|\nabla u(x)|^2dx\le\int_\Omega\left|\nabla\left(u^{*,\Omega}(x)\right)\right|^2dx.$$
\end{enumerate}
\begin{cor}$\dd{}$\\
Setting
$$\ttilde H(\Omega)=\bigg\{u\in H^1_0(\Omega):0\le u\text{ is even and nonincreasing in both variables}$$
$$\left.\text{with }\int_\Omega|\nabla u(x)|^2dx\le1\right\},$$
then
$$\sup_{u\in H^1_0(\Omega),\int_\Omega|\nabla u(x)|^2dx\le1}\int_\Omega\left(e^{4\pi u(x)^2}-1\right)dx=\sup_{u\in \ttilde H(\Omega)}\int_\Omega\left(e^{4\pi u(x)^2}-1\right)dx.$$
\end{cor}
Unlike what happens with Schwarz symmetrization, this does not restrict the problem to one-dimensional functions; however, some estimates similar to Lemmas $\dd{\ref{lemma:radial}}$ and $\dd{\ref{lemma:radialip}}$ hold for symmetrized functions, and they will be essential for our purpose:
\begin{lemma}
\label{lemma:stimastr}$\dd{}$\\
For any $\dd{u\in\ttilde H(\Omega)}$, it holds
\begin{equation}
\label{eq:stimastr}
u(x_1,x_2)^4\le f(x_1,x_2)=\left\{\begin{array}{ll}\left(\frac{1}{\sqrt{\lambda_1(\Omega)}}+1\right)^4\frac{1}{x_2^2}&\text{if }|x_1|\le1\\\frac{4}{x_1^2}&\text{if }|x_1|>1\end{array}\right.
\end{equation}
with $\dd{f\in L^1\left(\Omega\backslash B_\epsi(0)\right)}$ for any $\dd{\epsi>0}$.
\end{lemma}
\begin{proof}$\dd{}$\\
Clearly, it suffices to provide estimates for nonnegative $\dd{x_1}$ and $\dd{x_2}$; since any function in $\dd{H^1(\Omega)}$ is absolutely continuous along almost every line, the decreasing character of $\dd{u}$ gives, for $\dd{x_1>0}$,
$$u(x_1,x_2)=\frac{1}{x_1}\int_0^{x_1}u(x_1,x_2)ds\le\frac{1}{x_1}\int_0^{x_1}u(s,x_2)ds\le$$
$$\le\frac{1}{x_1}\int_0^{x_1}ds\int_{-1}^{x_2}\de_{x_2}u(s,t)dt\le$$
$$\le\frac{1}{x_1}\left(\int_0^{x_1}ds\int_{-1}^{x_2}dt\right)^\frac{1}2\left(\int_0^{x_1}ds\int_{-1}^{x_2}|\de_{x_2}u(s,t)|^2dt\right)^\frac{1}2\le$$
$$\le\frac{1}{x_1}\sqrt{x_1(x_2+1)}\|\nabla u\|_{L^2(\Omega)}\le\sqrt{\frac{2}{x_1}}\|\nabla u\|_{L^2(\Omega)}\le\sqrt{\frac{2}{x_1}}.$$
Moreover, for $\dd{x_1\in[0,1]}$ and $\dd{x_2>0}$,
$$u(x_1,x_2)=\frac{1}{x_2}\int_0^{x_2}u(x_1,x_2)dt\le\frac{1}{x_2}\int_0^{x_2}u(x_1,t)dt=$$
$$=\frac{1}{x_2}\int_0^{x_2}dt\left(u(1,t)-\int_{x_1}^1\de_{x_1}u(s,t)ds\right)dt\le$$
$$\le\frac{1}{x_2}\left(\int_0^{x_2}dt\int_0^1u(s,t)ds+\int_0^{x_2}ds\int_{x_1}^1|\de_{x_1}u(s,t)|dt\right)\le$$
$$\le\frac{1}{x_2}\left(\left(\int_0^{x_2}dt\int_0^1ds\right)^\frac{1}2\left(\int_0^{x_2}dt\int_0^1u^2(s,t)ds\right)^\frac{1}2+\right.$$
$$\left.+\left(\int_0^{x_2}dt\int_{x_1}^1ds\right)^\frac{1}2\left(\int_0^{x_2}dt\int_{x_1}^1|\nabla u(s,t)|^2ds\right)^\frac{1}2\right)\le$$
$$\le\frac{1}{x_2}\left(\sqrt{x_2}\|u\|_{L^2(\Omega)}+\sqrt{x_2(1-x_1)}\|\nabla u\|_{L^2(\Omega)}\right)\le$$
$$\le\frac{1}{\sqrt{x_2}}\left(\|u\|_{L^2(\Omega)}+\|\nabla u\|_{L^2(\Omega)}\right)\le\frac{1}{\sqrt{x_2}}\left(\frac{1}{\sqrt{\lambda_1(\Omega)}}+1\right)\|\nabla u\|_{L^2(\Omega)}\le$$
$$\le\frac{1}{\sqrt{x_2}}\left(\frac{1}{\sqrt{\lambda_1(\Omega)}}+1\right).$$
\end{proof}
This lemma implies that, away from the origin, the Moser-Trudinger functional is compact on $\dd{\ttilde H(\Omega)}$, except for the first term of its power series expansion:
\begin{cor}
\label{cor:stimastr}$\dd{}$\\
Let $\dd{u_k}$ be a sequence in $\dd{\ttilde H(\Omega)}$ such that $\dd{u_k\underset{k\to+\infty}\to u}$ almost everywhere in $\dd{\Omega}$.\\
Then
$$e^{4\pi u_k^2}-1-4\pi u_k^2\underset{k\to+\infty}\to e^{4\pi u^2}-1-4\pi u^2\quad\text{in }L^1(\Omega\backslash B_\epsi(0))\quad\forall\;\epsi>0.$$
\end{cor}
\begin{proof}$\dd{}$\\
Using the estimate $\dd{\eqref{eq:stimaesp}}$ and $\dd{\eqref{eq:stimastr}}$, one has
$$\left|e^{4\pi u_k^2}-1-4\pi u_k^2\right|\le8\pi^2|u_k|^4e^{4\pi u_k^2}\le C(\epsi)f\in L^1(\Omega\backslash B_\epsi(0)),$$
hence Lebesgue's dominated convergence theorem gives the claim.
\end{proof}
Another consequence of Lemma $\dd{\ref{lemma:stimastr}}$ is that when the limit is null, the only contributions in the integral come from neighborhoods of the origin and of infinity.
\begin{cor}
\label{cor:stimast2}$\dd{}$\\
Let $\dd{u_k}$ be a sequence in $\dd{\ttilde H(\Omega)}$ such that $\dd{u_k\underset{k\to+\infty}\to 0}$ almost everywhere in $\dd{\Omega}$.\\
Then, for any $\dd{\epsi,R>0}$, it holds
$$\int_{\Omega_R\backslash B_\epsi(0)}\left(e^{4\pi u_k^2(x)}-1\right)dx\underset{k\to+\infty}\to0$$
where $\dd{\Omega_R:=(-R,R)\times(-1,1)}$.
\end{cor}
\begin{proof}$\dd{}$\\
As before, one has
$$\left|e^{4\pi u_k^2}-1\right|\le4\pi u_k^2 e^{4\pi u_k^2}\le C(\epsi)\sqrt f\in L^\infty(\Omega_R\backslash B_\epsi(0)).$$
hence the results follows applying again Lebesgue's dominated convergence theorem on the bounded set $\dd{\Omega_R\backslash B_\epsi(0)}$.
\end{proof}
The proof of Theorem $\dd{\ref{teo:striscia}}$ will follow several steps: taking a maximizing sequence $\dd{u_k\in H^1_0(\Omega)}$, one may suppose that, up to subsequences, one has $\dd{u_k\underset{k\to+\infty}\to u}$ weakly in $\dd{H^1_0(\Omega)}$, strongly in $\dd{L^2_{\loc}(\Omega)}$ and almost everywhere in $\dd{\Omega}$; moreover, up to replacing $\dd{u_k}$ with $$\frac{|u_k|^{*,\Omega}}{\|\nabla |u_k|^{*,\Omega}\|_{L^2(\Omega)}},$$
it is not restrictive to take $\dd{u_k\in\ttilde H(\Omega)}$ and $\dd{\int_\Omega|\nabla u_k(x)|^2dx=1}$.\\
Denoting the supremum of the Moser-Trudinger functional as
$$S:=\sup_{u\in H^1_0(\Omega),\int_\Omega|\nabla u(x)|^2dx\le1}\int_\Omega\left(e^{4\pi u(x)^2}-1\right)dx$$
we will prove, in the order:
\begin{enumerate}
\item If $\dd{u\not\equiv0}$, then $\dd{S}$ is attained.
\item If $\dd{u\equiv0}$, set
$$\theta:=\frac{1}S\lim_{k\to+\infty}\int_{B_\epsi(0)}\left(e^{4\pi u_k(x)^2}-1\right)dx.$$
Then, by Corollary $\dd{\ref{cor:stimast2}}$, it is independent on $\dd{\epsi}$ and smaller than $\dd{1}$ (No concentration).
\item $\dd{\theta\notin(0,1)}$ (No dichotomy).
\item If $\dd{\theta=0}$ (Vanishing), then $\dd{S\le\frac{4\pi}{\lambda_1(\Omega)}}$.
\item $\dd{S>\frac{4\pi}{\lambda_1(\Omega)}}$.
\end{enumerate}
An important tool from complex analysis will turn out to be useful in this proof: the conformal diffeomorphism between the unit disc $\dd{B_1(0)\subset\mb R^2}$ and $\dd{\Omega}$, that is unique up to rotations, precisely
\begin{equation}
\label{eq:diffeo}
\psi(y_1,y_2)=\left(\frac{1}\pi\log\frac{(y_1-1)^2+y_2^2}{(y_1+1)^2+y_2^2},\frac{2}\pi\arctan\frac{2y_1}{1-y_1^2-y_2^2}\right).
\end{equation}
For any $\dd{u\in H^1_0(\Omega)}$ it holds
$$\int_\Omega\left(e^{4\pi u(x)^2}-1\right)dx=\int_{B_1(0)}\left(e^{4\pi(u\circ\psi)(y)^2}-1\right)dV_{g_\psi}(y),$$
where
$$g_\psi=|\det\de_y\psi|g_e=\frac{16}{\pi^2}\frac{1}{\left((y_1+1)^2+y_2^2\right)\left((y_1-1)^2+y_2^2\right)}g_e$$
is unbounded only around the two points $\dd{(\pm1,0)}$, that are mapped by $\dd{\psi}$ at infinity.\\
Moreover, the conformal invariance of the Dirichlet integral ensures that the sequence $\dd{v_k=u_k\circ\psi}$ is bounded in $\dd{H^1_0(B_1(0),g_\psi)}$, hence it must converge to $\dd{v=u\circ\psi}$ weakly, strongly in $\dd{L^2_{\loc}(B_1(0)\backslash\{(\pm1,0)\},g_\psi)}$ and a.e. on $\dd{B_1(0)}$.\\
Following what was done by Moser on bounded domains \cite{moser} and using the hyperbolic Moser-Trudinger inequality $\dd{\ref{teo:conformi}}$, we get:\\
\begin{lemma}
\label{lemma:mtsopra}$\dd{}$\\
For any $\dd{u\in H^1_0(\Omega)}$, $\dd{\alpha>0}$ it holds
$$\int_{\Omega}\left(e^{\alpha u^2(x)}-1\right)dx<+\infty$$
\end{lemma}
\begin{proof}$\dd{}$\\
We apply the conformal diffeomorphism with the unit disc $\dd{\eqref{eq:diffeo}}$; since the metric $\dd{g_\psi}$ generated by this map is bounded from above by the hyperbolic one, it suffices to verify
$$\int_{B_1(0)}\left(e^{\alpha v^2(y)}-1\right)dV_{g_h}(y)<+\infty\quad\forall\;v\in H^1_0(B_1(0),g_h).$$
Moreover, applying hyperbolic symmetrization, we can restrict our proof to nonnegative radially decreasing functions (see \cite{mansand} for details), for which Lemma $\dd{\ref{lemma:radialip}}$ holds; therefore, keeping in mind estimate $\dd{\eqref{eq:stimaesp}}$ and the integrability of $\dd{e^{\alpha v^2}}$ on any bounded Euclidean sets, we get
$$\int_{B_1(0)}\left(e^{\alpha v^2(y)}-1\right)dV_{g_h}(y)\le$$
$$\le\int_{B_\frac{1}2(0)}e^{\alpha v^2(y)}dV_{g_h}(y)+\int_{B_1(0)\backslash B_\frac{1}2(0)}\left(e^{\alpha v^2(y)}-1\right)dV_{g_h}(y)\le$$
$$\le\frac{64}9\int_{B_\frac{1}2(0)}e^{\alpha v^2(y)}dV_{g_e}(y)+\int_{B_1(0)\backslash B_\frac{1}2(0)}v^2(y)e^{\alpha v^2(y)}dV_{g_h}(y)\le$$
$$\le C(v)+\int_{B_1(0)\backslash B_\frac{1}2(0)}\frac{v^2(y)}{|y|^{\frac{\alpha}\pi\|\nabla_{g_h}v\|_{L^2(B_1(0),g_h)}^2}}dV_{g_h}(y)\le$$
$$\le C(v)+C_1(v)\int_{B_1(0)}v^2(y)dV_{g_h}(y)\le C_2(v)$$
that is what we claimed.
\end{proof}
Some of the calculations used in this proof will be inspired by a work of Gi. Mancini and Sandeep \cite{mansand2}, who studied existence of extremals for the Moser-Trudinger inequality on the hyperbolic disc.
\begin{proof}[Proof of step $\dd{1}$]$\dd{}$\\
The first case that will be considered is when the convergence is strong in $\dd{L^2(\Omega)}$; under this hypothesis, it will be shown that $\dd{u}$ itself attains the supremum, that is
$$\lim_{k\to+\infty}\int_\Omega\left(e^{4\pi u_k(x)^2}-1\right)dx=\int_\Omega\left(e^{4\pi u(x)^2}-1\right)dx.$$
The modified concentration-compactness Theorem $\dd{\ref{teo:concomp2}}$ gives, for any $\dd{\Omega_R}$ defined as in Corollary $\dd{\ref{cor:stimast2}}$,
$$\lim_{k\to+\infty}\int_{\Omega_R}\left(e^{4\pi u_k(x)^2}-1\right)dx=\int_{\Omega_R}\left(e^{4\pi u(x)^2}-1\right)dx$$
whereas, near the infinity, one uses strong convergence for the quadratic term and Corollary $\dd{\ref{cor:stimastr}}$ for the rest:
$$\int_{\Omega\backslash\Omega_R}\left(e^{4\pi u_k(x)^2}-1\right)dx=$$
$$=\int_{\Omega\backslash\Omega_R}\left(e^{4\pi u_k(x)^2}-1-4\pi u_k^2(x)\right)dx+4\pi\int_{\Omega\backslash\Omega_R}u_k^2(x)dx\underset{k\to+\infty}\to$$
$$\underset{k\to+\infty}\to\int_{\Omega\backslash\Omega_R}\left(e^{4\pi u(x)^2}-1-4\pi u^2(x)\right)dx+4\pi\int_{\Omega\backslash\Omega_R}u^2(x)dx=$$
$$\int_{\Omega\backslash\Omega_R}\left(e^{4\pi u(x)^2}-1\right)dx.$$
If instead the convergence is just weak, a little more calculations are required; in this case, it holds
$$1-\sigma:=\int_\Omega|\nabla u(x)|^2dx\in(0,1),$$
since $\dd{\sigma=1}$ would mean that the weak limit is null and $\dd{\sigma=0}$ would give strong convergence in $\dd{H^1_0(\Omega)}$ and so, by Poincaré's inequality, also in $\dd{L^2(\Omega)}$; moreover, weak convergence gives
$$\sigma_k:=\int_\Omega|\nabla(u_k(x)-u(x))|^2dx=$$
$$=\int_\Omega|\nabla u_k(x)|^2dx-2\int_\Omega\langle\nabla u_k(x),\nabla u(x)\rangle dx+\int_\Omega|\nabla u(x)|^2dx$$
and so
$$\sigma_k\underset{k\to+\infty}\to\sigma\in(0,1).$$
Convexity of $\dd{f(t)=e^{4\pi t}-1}$ gives
$$\int_\Omega\left(e^{4\pi u_k(x)^2}-1\right)dx=$$
$$=\int_\Omega\left(e^{4\pi\left(\sigma_k\left(\frac{u_k(x)-u(x)}{\sqrt{\sigma_k}}\right)^2+(1-\sigma_k)\frac{2u(x)(u_k(x)-u(x))+u^2(x)}{1-\sigma_k}\right)}-1\right)dx\le$$
$$\le\sigma_k\int_\Omega\left(e^{4\pi\left(\frac{u_k(x)-u(x)}{\sqrt{\sigma_k}}\right)^2}-1\right)dx+$$
$$+(1-\sigma_k)\int_\Omega\left(e^{\frac{8\pi u(x)(u_k(x)-u(x))+4\pi u^2(x)}{1-\sigma_k}}-1\right)dx\le$$
$$\le\sigma_kS+(1-\sigma_k)\int_\Omega\left(e^{\frac{8\pi u(x)(u_k(x)-u(x))+4\pi u^2(x)}{1-\sigma_k}}-1\right)dx.$$
Passing to the limit for $\dd{k\to+\infty}$, the last inequality becomes
$$S\le\sigma S+(1-\sigma)\limsup_{k\to+\infty}\int_\Omega\left(e^{\frac{8\pi u(x)(u_k(x)-u(x))+4\pi u^2(x)}{1-\sigma_k}}-1\right)dx.$$
So, if we could pass to the limit inside the integral, we would get
\begin{equation}
\label{eq:esse}
S\le\sigma S+(1-\sigma)\int_\Omega\left(e^{4\pi\frac{u^2(x)}{1-\sigma}}-1\right)dx\le\sigma S+(1-\sigma)S=S,
\end{equation}
hence every inequality would have to be actually an equality, and thus $\dd{\frac{u}{\sqrt{1-\sigma}}}$ would be an extremal function.\\
Now we prove that one can actually take the limit inside the integral, studying separately what happens inside and outside $\dd{\Omega_R}$. In the rectangles, we find boundedness in $\dd{L^2}$ thanks to Lemma $\dd{\ref{lemma:mtsopra}}$, hence we can pass to the limit by applying Vitali's convergence theorem:
$$\int_{\Omega_R}\left(e^{\frac{8\pi u(x)(u_k(x)-u(x))+4\pi u^2(x)}{1-\sigma_k}}-1\right)^2dx\le$$
$$\le\int_{\Omega_R}e^\frac{16\pi u(x)|u_k(x)-u(x)|}{1-\sigma_k}e^\frac{8\pi u^2(x)}{1-\sigma_k}dx+4R\le$$
$$\le\left(\int_{\Omega_R}e^\frac{32\pi u(x)|u_k(x)-u(x)|}{1-\sigma_k}dx\right)^\frac{1}2\left(\int_{\Omega_R}e^\frac{16\pi u^2(x)}{1-\sigma_k}dx\right)^\frac{1}2+4R\le$$
$$\le\left(\int_{\Omega_R}e^{\frac{16\pi}{1-\sigma_k}\left(\frac{u(x)^2}\epsi+\epsi(u_k(x)-u(x))^2\right)}dx\right)^\frac{1}2C(u,R)+4R\le$$
$$\le\left(\int_{\Omega_R}e^{\frac{32\pi}{\epsi(1-\sigma_k)}u(x)^2}dx\right)^\frac{1}4\left(\int_{\Omega_R}e^{\frac{32\epsi\pi}{1-\sigma_k}(u_k(x)-u(x))^2}dx\right)^\frac{1}4C(u,R)+4R\le$$
$$\le\left(\int_{\Omega_R}e^{\frac{32\epsi\pi}{1-\sigma_k}(u_k(x)-u(x))^2}dx\right)^\frac{1}4C_1(u,R)+4R$$
which is uniformly bounded if one takes
$$\epsi<\frac{1-\sigma}{8\sigma}.$$
On the other hand, the integral on $\dd{\Omega\backslash\Omega_R}$ is small for large $\dd{R}$: the estimate $\dd{\eqref{eq:stimaesp}}$ and the uniform boundedness of $\dd{8\pi u(u_k-u)+4\pi u^2}$, which follows from estimate $\dd{\eqref{eq:stimastr}}$, give
\begin{equation}
\label{eq:omegar}
\left|\int_{\Omega\backslash\Omega_R}\left(e^{\frac{8\pi u(x)(u_k(x)-u(x))+4\pi u^2(x)}{1-\sigma_k}}-1\right)dx\right|\le
\end{equation}
$$\le\int_{\Omega\backslash\Omega_R}\max\left\{e^{\frac{8\pi u(x)(u_k(x)-u(x))+4\pi u^2(x)}{1-\sigma_k}},1\right\}\left|\frac{8\pi u(x)(u_k(x)-u(x))+4\pi u^2(x)}{1-\sigma_k}\right|dx\le$$
$$\le C(u)\int_{\Omega\backslash\Omega_R}\left(2u(x)|u_k(x)-u(x)|+u^2(x)\right)dx\le$$
$$\le 2C(u)\left(\int_\Omega(u(x)-u_k(x))^2dx\right)^\frac{1}2\left(\int_{\Omega\backslash\Omega_R}u^2(x)dx\right)^\frac{1}2+C(u)\int_{\Omega\backslash\Omega_R}u^2(x)dx\le$$
$$\frac{2C(u)}{\sqrt{\lambda_1(\Omega)}}\left(\int_{\Omega\backslash\Omega_R}u^2(x)dx\right)^\frac{1}2+\int_{\Omega\backslash\Omega_R}u^2(x)dx\le C(u)$$
which goes to $\dd{0}$ as $\dd{R}$ goes to $\dd{+\infty}$; therefore, given $\dd{\epsi>0}$, one can find $\dd{R}$ such that $\dd{\eqref{eq:omegar}}$ is smaller than $\dd{\epsi}$ and, using convergence in $\dd{L^1(\Omega_R)}$, we get
$$S\le\sigma S+(1-\sigma)\left(\int_{\Omega_R}\left(e^{4\pi\frac{u^2(x)}{1-\sigma}}-1\right)+\epsi\right)dx\le\sigma S+(1-\sigma)(S+\epsi)$$
which is, being $\dd{\epsi}$ arbitrary, $\dd{\eqref{eq:esse}}$.
\end{proof}
\begin{proof}[Proof of step $\dd{2}$]$\dd{}$\\
If $\dd{\theta=1}$, then for any $\dd{\epsi>0}$ it holds
$$S=\lim_{k\to+\infty}\int_{B_\epsi(0)}\left(e^{4\pi u_k(x)^2}-1\right)dx=$$
$$=\lim_{k\to+\infty}\int_{\psi^{-1}(B_\epsi(0))}\left(e^{4\pi(u_k\circ\psi)(y)^2}-1\right)|\det\de_y\psi(y)|dy\le$$
$$\le\sup_{\psi^{-1}(B_\epsi(0))}|\det\de_y\psi(y)|\lim_{k\to+\infty}\int_{B_1(0)}\left(e^{4\pi(u_k\circ\psi)(y)^2}-1\right)dy\le$$
$$\le\sup_{\psi^{-1}(B_\epsi(0))}|\det\de_y\psi(y)|\sup_{\int_{B_1(0)}|\nabla v(x)|^2dx\le1}\int_{B_1(0)}\left(e^{4\pi v^2(y)}-1\right)dy.$$
Passing to the limit for $\dd{\epsi\to0}$, $\dd{\psi^{-1}(B_\epsi(0))}$ shrinks around $\dd{0}$, so
$$S\le|\det\de_y\psi(0)|\sup_{\int_{B_1(0)}|\nabla v(x)|^2dx\le1}\int_{B_1(0)}\left(e^{4\pi v^2(y)}-1\right)dy.$$
Since $\dd{\psi}$ is a (nonlinear) conformal diffeomorphism, $\dd{|\det\de_y\psi|}$ is (strictly) subharmonic, hence one can apply the (strict) mean value inequality; thus, if
$$\ttilde u(x)=\ttilde U(|x|)$$
is a radial extremal for the Moser-Trudinger inequality on the unit disc,
$$S\le|\det\de_y\psi(0)|\int_{B_1(0)}\left(e^{4\pi{\ttilde u(y)}^2}-1\right)dy=$$
$$=2\pi|\det\de_y\psi(0)|\int_0^1\left(e^{4\pi{\ttilde U(\rho)}^2}-1\right)\rho d\rho<$$
$$<\int_0^1\left(e^{4\pi{\ttilde U(\rho)}^2}-1\right)\rho d\rho\int_{-\pi}^\pi|\det\de_y\psi(\rho\cos\theta,\rho\sin\theta)|d\theta=$$
$$=\int_{B_1(0)}\left(e^{4\pi\ttilde u(y)^2}-1\right)dV_{g_\psi}(y)=\int_\Omega\left(e^{4\pi\left(\ttilde u\circ\psi^{-1}\right)(x)^2}-1\right)dx\le S,$$
which is a contradiction.
\end{proof}
\begin{oss}$\dd{}$\\
The proof of step $\dd{2}$ closely follows the estimates in \cite{flucher} for the concentration level in the case of planar simply connected domain; moreover, we did not use any of the symmetry properties of $\dd{\Omega}$, which are instead crucial in most of the rest of the proof of Theorem $\dd{\ref{teo:striscia}}$: actually, the proof of this step can be reproduced for any domain which is conformally equivalent to the ball and where the Moser-Trudinger inequality holds.\\
This can be intuitively explained by the fact that concentration is a local property, hence it does not depend on the shape of the domain.
\end{oss}
\begin{proof}[Proof of step $\dd{3}$]$\dd{}$\\
Fixed $\dd{\epsi,R>0}$, one takes a test function $\dd{\pphi\in C^1_0(\Omega)}$ such that
$$\left\{\begin{array}{ll}0\le\pphi\le1\\\pphi\equiv1&\text{on }B_{\epsi}(0)\\\pphi\equiv0&\text{on }\Omega\backslash\Omega_R\end{array}\right.;$$
then one has
$$\theta S=\lim_{k\to+\infty}\int_{B_\epsi(0)}\left(e^{4\pi u_k(x)^2}-1\right)dx=$$
\begin{equation}
\label{eq:tetas}
=\lim_{k\to+\infty}\int_{B_\epsi(0)}\left(e^{4\pi(u_k(x)\pphi(x))^2}-1\right)dx
\end{equation}
and
$$(1-\theta)S=\lim_{k\to+\infty}\int_{\Omega\backslash\Omega_R}\left(e^{4\pi u_k(x)^2}-1\right)dx=$$
$$=\lim_{k\to\infty}\int_{\Omega\backslash\Omega_R}\left(e^{4\pi(u_k(x)(1-\pphi(x)))^2}-1\right)dx.$$
Moreover,
$$1=\int_\Omega|\nabla u_k(x)|^2dx=\int_\Omega|\nabla(u_k(x)\pphi(x))|^2dx+\int_\Omega|\nabla(u_k(x)(1-\pphi(x)))|^2+$$
$$+2\int_\Omega\langle\nabla(u_k(x)\pphi(x)),\nabla(u_k(x)(1-\pphi(x)))\rangle dx$$
and strong convergence in $\dd{L^2_{\loc}(\Omega)}$ yields
$$\int_\Omega\langle\nabla(u_k(x)\pphi(x)),\nabla(u_k(x)(1-\pphi(x)))\rangle dx=$$
$$=\int_{\Omega_R\backslash B_\epsi(0)}\langle(\pphi(x)\nabla u_k(x)+u_k(x)\nabla\pphi(x)),((1-\pphi(x))\nabla u_k(x)-u_k(x)\nabla\pphi(x))\rangle dx=$$
$$=\int_{\Omega_R\backslash B_\epsi(0)}\left(\pphi(x)(1-\pphi(x))|\nabla u_k(x)|^2+u_k(x)(1-\pphi(x))\langle\nabla u_k(x),\nabla\pphi(x)\rangle-\right.$$
$$\left.-u_k(x)\pphi(x)\langle\nabla u_k(x),\nabla\pphi(x)\rangle-u_k(x)^2|\nabla\pphi(x)|^2\right)dx=$$
$$=\int_{\Omega_R\backslash B_\epsi(0)}\pphi(x)(1-\pphi(x))|\nabla u_k(x)|^2dx+o(1)\ge o(1),$$
hence
$$\int_\Omega|\nabla(u_k(x)(1-\pphi(x)))|^2dx+\int_\Omega|\nabla(u_k(x)\pphi(x))|^2\le 1.$$
If it were
$$\int_\Omega|\nabla(u_k(x)(1-\pphi(x)))|^2dx\underset{k\to+\infty}\to0$$
then Poincaré's inequality would give 
$$u_k(1-\pphi)\underset{k\to+\infty}\to0\quad\text{in }L^2(\Omega),$$
and applying Corollary $\dd{\ref{cor:stimastr}}$ to $\dd{u_k(1-\pphi)}$, we find
$$\lim_{k\to+\infty}\int_{\Omega\backslash\Omega_R}\left(e^{4\pi(u_k(x)(1-\pphi(x)))^2}-1\right)dx=$$
$$=\lim_{k\to+\infty}\int_{\Omega\backslash\Omega_R}\left(e^{4\pi(u_k(x)(1-\pphi(x)))^2}-1-4\pi(u_k(x)(1-\pphi(x)))^2\right)dx+$$
$$+4\pi\lim_{k\to+\infty}\int_{\Omega\backslash\Omega_R}(u_k(x)(1-\pphi(x)))^2dx=0$$
that is $\dd{\theta=1}$.\\
Otherwise it must be
$$\tau:=\limsup_{k\to+\infty}\int_\Omega|\nabla(u_k(x)\pphi(x))|^2dx<1$$
but then one might apply Vitali's theorem since, for $\dd{p\in\left(1,\frac{1}\tau\right)}$, one has
$$\int_{B_\epsi(0)}\left(e^{4\pi(u_k(x)\pphi(x))^2}-1\right)^pdx\le$$
$$\int_{B_\epsi(0)}e^{4p\pi\left\|\nabla(u_k\pphi)\right\|_{L^2(\Omega)}^2\left(\frac{u_k(x)\pphi(x)}{\left\|\nabla(u_k\pphi)\right\|_{L^2(\Omega)}}\right)^2}dx\le S,$$
so one gets
$$\lim_{k\to+\infty}\int_{B_\epsi(0)}\left(e^{4\pi(u_k(x)\pphi(x))^2}-1\right)dx=0$$
hence, from $\dd{\eqref{eq:tetas}}$ it must be $\dd{\theta=0}$.
\end{proof}
\begin{proof}[Proof of step $\dd{4}$]$\dd{}$\\
It is a straightforward consequence of Corollary $\dd{\ref{cor:stimastr}}$:
$$S=\lim_{k\to+\infty}\int_{\Omega\backslash\Omega_R}\left(e^{4\pi u_k(x)^2}-1\right)dx=$$
$$=\lim_{k\to+\infty}\int_{\Omega\backslash\Omega_R}\left(e^{4\pi u_k(x)^2}-1-4\pi u_k(x)^2\right)dx+$$
$$+4\pi\lim_{k\to+\infty}\int_{\Omega\backslash\Omega_R}u_k(x)^2dx=4\pi\lim_{k\to+\infty}\int_{\Omega\backslash\Omega_R}u_k(x)^2dx\le\frac{4\pi}{\lambda_1(\Omega)}.$$
\end{proof}
\begin{proof}[Proof of step $\dd{5}$]$\dd{}$\\
An explicit calculation, by separation of variables, of the first eigenvalues and eigenfunctions of $\dd{-\Delta}$ on $\dd{\Omega_k}$ shows that $\dd{\lambda_1(\Omega)=\frac{\pi^2}4}$, so it suffices to show that the functional can assume values strictly larger than $\dd{\frac{16}\pi}$.\\
The first eigenfunctions on $\dd{\Omega_k}$ are respectively
$$\pphi_k(x_1,x_2)=\left\{\begin{array}{ll}\cos\left(\frac{\pi}{2k}x_1\right)\cos\left(\frac{\pi}2x_2\right)&\text{if }|x_1|\le k\\0&\text{if }|x_1|>k\end{array}\right.,$$
so
$$u_k:=\frac{\pphi_k}{\|\nabla\pphi_k\|_{L^2(\Omega)}}$$
is a vanishing sequence and the evaluation of the Moser-Trudinger functional on $\dd{u_k}$ tends to the critical value $\dd{\frac{16}\pi}$, as can be seen by the proof of step $\dd{5}$; we will now show that this value is reached from above.\\
It holds
$$\int_\Omega\pphi_k(x_1,x_2)^2dx_1dx_2=\int_{-k}^k\cos\left(\frac{\pi}{2k}x_1\right)^2dx_1\int_{-1}^1\cos\left(\frac{\pi}2x_2\right)^2dx_2=k$$
and
$$\|\nabla\pphi_k\|_{L^2(\Omega)}^2=$$
$$=\int_{-k}^kdx_1\int_{-1}^1dx_2\left|\left(-\frac{\pi}{2k}\sin\left(\frac{\pi}{2k}x_1\right)\cos\left(\frac{\pi}2x_2\right),-\frac{\pi}2\cos\left(\frac{\pi}{2k}x_1\right)\sin\left(\frac{\pi}2x_2\right)\right)\right|^2=$$
$$=\frac{\pi^2}{4k^2}\int_{-k}^k\sin\left(\frac{\pi}{2k}x_1\right)^2dx_1\int_{-1}^1\cos\left(\frac{\pi}2x_2\right)^2dx_2+$$
$$+\frac{\pi^2}4\int_{-k}^k\cos\left(\frac{\pi}{2k}x_1\right)^2dx_1\int_{-1}^1\sin\left(\frac{\pi}2x_2\right)^2dx_2=\frac{\pi^2}4\frac{k^2+1}k,$$
hence, by Jensen's inequality,
$$\int_\Omega\left(e^{4\pi u_k(x)^2}-1\right)dx=\int_{\Omega_k}\left(e^{4\pi\left(\frac{\pphi_k(x)}{\|\nabla\pphi_k\|_{L^2(\Omega)}}\right)^2}-1\right)dx\ge$$
$$\ge|\Omega_k|\left(e^{4\pi\frac{\int_{\Omega_k}\pphi_k(x)^2dx}{|\Omega_k|\|\nabla\pphi_k\|_{L^2(\Omega)}^2}}-1\right)=4k\left(e^{\frac{4}\pi\frac{k}{k^2+1}}-1\right)$$
that goes to $\dd{\frac{16}\pi}$ from above; therefore, the vanishing value is exceeded by eigenfunctions of large rectangles.
\end{proof}
Some of the results obtained can be extended to other simply connected domains which satisfy some hypotheses; the double symmetrization can be made in any domain $\dd{\Omega}$ which coincides with its double symmetrized, for instance if
$$\Omega_f=\{(x_1,x_2)\in\mb R^2:|x_2|<f(|x_1|)\}$$
with $\dd{f\in C^1((0,+\infty),(0,+\infty))}$ nonincreasing (and actually the strip we considered in Theorem $\dd{\ref{teo:striscia}}$ is a domain of this kind with $\dd{f\equiv1}$). Probably something similar can be done if some other kind of symmetry holds, but it is not known if this implies some estimates like the ones in Lemma $\dd{\ref{lemma:stimastr}}$, which were essential in most of the proof.\\
Another difficulty is that the precise value of $\dd{\lambda_1(\Omega)}$, that was needed to find functions above the vanishing level, is generally not known except in very special cases; this problem might be bypassed if a first eigenfunction $\dd{\pphi}$ exists for the laplacian, since in this case
$$\int_\Omega\left(e^{4\pi\left(\frac{\pphi(x)}{\|\nabla\pphi\|_{L^2(\Omega)}}\right)^2}-1\right)dx>4\pi\int_\Omega\frac{\pphi^2(x)}{\|\nabla\pphi\|_{L^2(\Omega)}^2}dx=\frac{4\pi}{\lambda_1(\Omega)}.$$
However, this condition is difficult to be verified, because the embedding of $\dd{L^2(\Omega)}$ in $\dd{H^1_0(\Omega)}$ is generally not compact, even if $\dd{\Omega}$ has finite measure (see for instance \cite{adamslib}).\\
Finally, extenstions to strips in higher dimension seem difficult as well, because one has to deal with the $\dd{N}$-laplacian $\dd{\Delta_Nu=\dive(|\nabla u|^{N-2}\nabla u)}$ and much less is known about its spectral properties than for the usual Laplace operator; moreover, the domains which are conformally equivalent to the ball are much less than in the planar case.\\

\section*{Acknowledgements}
We would like to thank our Master's Degree supervisor Professor Gianni Mancini, who led us to the discovery of the these results during the preparation of our theses and encouraged us to write this paper.\\
We also thank Professor Andrea Malchiodi who kindly offered to check our work and corrected several mistakes.

\bibliography{mydatabase}
\bibliographystyle{alpha}

\end{document}